\documentclass[twoside]{amsart}
\usepackage{amsmath,amssymb,amscd,amsthm,amsxtra,amsfonts, pxfonts}
\usepackage{mathtools,mathrsfs,dsfont,xparse}
\usepackage{graphicx,epstopdf}
\usepackage{multirow}
\usepackage{cases}
\usepackage{enumerate}
\usepackage{xcolor}
\usepackage{tikz,pgfplots}
\usetikzlibrary{matrix}
\usepackage{mathtools}
\usepackage{todonotes}
\usepackage[margin=1.05 in]{geometry}

\mathtoolsset{showonlyrefs}

\allowdisplaybreaks[4]

\date{}

\usepackage{multicol,bbm}
\usepackage[colorlinks=true, pdfstartview=FitV, linkcolor=blue, citecolor=blue, urlcolor=blue]{hyperref}

\numberwithin{equation}{section}

\usepackage{cjhebrew}

\DeclareMathOperator{\im}{Im}

\usepackage[numbers,sort]{natbib}

{\theoremstyle {definition} \newtheorem {defn} {Definition} [section] }
{\theoremstyle {plain}  \newtheorem {theorem} [defn] {Theorem}}
{\theoremstyle {plain}  \newtheorem {corollary} [defn]{Corollary}}
{\theoremstyle {plain} \newtheorem {proposition} [defn]{Proposition}}
{\theoremstyle {plain} \newtheorem {lemma}[defn] {Lemma}}
{\theoremstyle {definition} \newtheorem {remark}[defn] {Remark}}
{\theoremstyle {plain} \newtheorem {claim}[defn] {Claim}}

\def\Z{{\mathbb{Z}}}
\def\T{{\mathbb{T}}}
\def\R{{\mathbb{R}}}
\def\C{{\mathbb{C}}}
\def\N{{\mathbb{N}}}

\newcommand{\abs}[1]{\left|#1\right|}
\newcommand{\norm}[1]{\left\|#1\right\|}
\newcommand{\inner}[1]{\left \langle#1\right \rangle}

\newcommand{\parenthese}[1]{\left(#1\right)}
\newcommand{\bracket}[1]{\left\{#1\right\}}

\begin{document}

\author[Wilson and Yu]{Bobby Wilson  and Xueying Yu}

\address{Bobby Wilson
\newline \indent Department of Mathematics, University of Washington\indent 
\newline \indent  C138 Padelford Hall Box 354350, Seattle, WA 98195,\indent }
\email{blwilson@uw.edu}

\address{Xueying  Yu
\newline \indent Department of Mathematics, University of Washington\indent 
\newline \indent  C138 Padelford Hall Box 354350, Seattle, WA 98195,\indent }
\email{xueyingy@uw.edu}

\title[Modified Scattering of Cubic Nonlinear Schr\"odinger Equation on Rescaled Waveguide Manifolds]{Modified Scattering of Cubic Nonlinear Schr\"odinger Equation on Rescaled Waveguide Manifolds}



\begin{abstract}
We use modified scattering theory to demonstrate that small-data solutions to the cubic nonlinear Schr\"odinger equation on rescaled waveguide manifolds, $\mathbb{R} \times \mathbb{T}^d$ for $d\geq 2$,  demonstrate boundedness of Sobolev norms as well as weak instability.
\end{abstract}

\maketitle

\setcounter{tocdepth}{1}
\tableofcontents

\parindent = 10pt     
\parskip = 8pt

\section{Introduction}

In this paper, we consider the following cubic  nonlinear Schr\"odinger (NLS) equation posed on the ``irrational" rescaling of the  waveguide manifolds $\R \times \T_{\theta}^d$: 
\begin{align}\label{eq NLS}
(i\partial_t + \Delta_{\R \times \T_{\theta}^d}) \, U = \mu \abs{U}^2 U, \quad (x,y) \in \R \times \T_{\theta}^d ,
\end{align}
where $\mu = \pm 1$, $\T_{\theta}^d$ is an irrational torus, and $U$ is a complex-valued function on the spatial domain $(x, y) \in \R \times \T_{\theta}^d$. Here $\theta=(\theta_1, ..., \theta_d) \in \mathbb{R}_+^d$ and the $d$-dimensional torus re-scaled by the vector $\theta$ is defined by
    \begin{align*}
        \T_{\theta}^d:= \prod_{i=1}^d \tfrac{1}{\theta_i}\mathbb{T} ,
    \end{align*}
    where $\mathbb{T}:= \mathbb{R}/(2\pi \mathbb{Z})$. We will say that $\theta$ is ``irrational" if the equation
    \begin{align*}
        \sum_{i=1}^d n_i \theta_i^2 =0
    \end{align*}
    does not have any nontrivial solutions, $(n_1, ..., n_d) \in \mathbb{Z}^d$. More accurately, the squares of the components of $\theta$ are irrational.

The NLS \eqref{eq NLS} conserves the Hamiltonian defined as follows
\begin{align}
{H}(U(t))  = \int_{\mathbb{R}\times \mathbb{T}_{\theta}^d} \tfrac{1}{2} |\nabla U(t,x,y)|^2  +   \tfrac{\mu}{4}  |U(t,x,y)|^{4} \, dx dy .
\end{align}

For simplicity, the body of this paper is presented using the defocusing version ($\mu=1$) of the NLS system, \eqref{eq NLS}. 

\subsection{Motivation and Background}

The waveguide manifold $\R^m \times \T^n$ is a product of the Euclidean space and the tori, and of particular interest in nonlinear optics of telecommunications, see for example \cite{B08}. Due to the nature of such product spaces, we see NLS posed on the waveguide manifold mixed inheriting properties from those on classical Euclidean spaces and tori.

The goal of this work is to study the asymptotic behavior of \eqref{eq NLS}. In fact, this project is inspired by Hani, Pausader, Tzvetkov, and Visciglia \cite{HPTV15}, where the authors investigated the asymptotic behavior of the cubic NLS posted on ``rational" waveguide $\R \times \T^d$, where $\T^d$ is a rational torus, and showed modified scattering for solutions to such equation with small initial data. 

Before continuing on to our motivation, let us spend a moment to explain different long time behaviors that one would expect on Euclidean spaces and tori, such as scattering phenomenon, growth of high Sobolev norms and modified scattering. In general terms,  scattering is the behavior by which nonlinear solutions converge to linear solutions as time approaches  infinity. Such scattering effect is expected to hold on most noncompact Riemannian manifolds (see, for example, \cite{CKSTT08,B99, KM06, D12}), while on bounded domains, especially on the torus, one anticipates that the high Sobolev norms of solution will grow over time. Such growth is very much related to the phenomenon of weak turbulence which is  described as the  transference of  energy from low frequencies of a solution to high frequencies of a solution, causing high Sobolev norms to grow while the energy of the solution remains bounded. The first of such constructions   appears in Colliander, Keel, Staffilani, Takaoka, and Tao \cite{CKSTT10}.
As for the ``modified scattering", one attempts to relate the solution to the full equation to an effective or resonant system whose behavior may be much different than the behavior of the linear flow.  In fact, the effective system used in   \cite{HPTV15}  is shown to exhibit behavior much different than that of the linear system. A class of noncompact manifolds in which different elements exhibit both weak turbulence and scattering are the waveguide manifolds. For example, on $\mathbb{R}^d \times \mathbb{T}$, $d \geq 1$, Tzvetkov and Visciglia, \cite{TV16}, prove that the cubic NLS system exhibits scattering regardless of the size of the initial data. On the other hand, in \cite{HPTV15}, cubic NLS on $\mathbb{R} \times \mathbb{T}^d$, $1 \leq d \leq 4$, is shown not to exhibit scattering.

The second important source of motivation for this project is the study of the dynamics of cubic NLS defined on irrational tori: 
\begin{align}\label{eq TNLS}
    (i\partial_t + \Delta_{ \T_{\theta}^d}) \, u = \abs{u}^2 u, \quad x \in  \T_{\theta}^d.
\end{align}
In recent works \cite{SW20} and \cite{HPSW21}, the authors studied the growth of high Sobolev norms on irrational tori and compared the long time behavior of such norms with those in the rational torus setting. 
The works suggest that there exists stronger stability of the cubic NLS in the irrational case than what one expects in the square case. Particularly in light of the work on instability by the I-team for the square torus case, \cite{CKSTT10}. One could also consult the work of Deng, Germain, and Guth  \cite{DGG17} on improved long-time Strichartz estimates to gain the same intuition. However, the picture is not so simple. In the paper of Giuliani and Guardia, \cite{GG22}, they show that the type of quantitative instability results given by \cite{GK15} hold (to a lesser extent) in the case of most irrational rescalings of the torus.

Motivated by works \cite{SW20} and \cite{HPSW21} by the first author of this paper and collaborators, we would like to consider the analogue of \cite{HPTV15} in the ``irrational" waveguide setting .

\subsection{Main Result and Discussion}

Now let us  present the main result in this paper:
\begin{theorem}\label{thm.main}
    Let $d\geq 2$. There exists $ \Theta_2 \subset \Theta_1 \subset \mathbb{R}^d_+$  such that the following hold:
        \begin{enumerate}
            \item  If $\theta \in \Theta_1$,  one can define $N_{\theta}\geq 0$ such that for each $N\geq 10d+N_{\theta}$ there exists $\varepsilon(N, d)>0$ with the following property:  If $U_0 \in S^+$ and $\|U_0\|_{S^+}< \varepsilon(N, d)$, there exists a solution to \eqref{eq NLS}, $U(t)$, such that $U(0)=U_0$ and 
                \begin{align}
                    \sup_{s<N}\sup_t \|U(t)\|_{H^s} <\infty .
                \end{align}
               
            \item If $\theta \in \Theta_2$ and $C>0$, there exists a solution to \eqref{eq NLS}, $V(t)$, and $T_2>T_1>0$ such that
                \begin{align}
                    \|V(T_2)\|_{H^N}\geq C \|V(T_1)\|_{H^N} .
                \end{align}
                and $T_2-T_1 \leq \exp(c_N\exp(C^{\beta})/\varepsilon^2)$ where $\beta=\beta(N)$ and $N\geq 10d+N_{\theta}$.
            \item 
                \begin{align}
                    \mathcal{L}^d(\mathbb{R}^d_+ \setminus \Theta_2)&=0,\\
                    \dim_{\mathcal{H}}(\mathbb{R}^d_+ \setminus \Theta_1)&=0.
                \end{align}
            
        \end{enumerate}
\end{theorem}

Here $\dim_{\mathcal{H}}(E)$ represents the Hausdorff dimension of the set $E$. The form of this statement is inspired by the work of G\'{e}rard and Grellier \cite{GG10} in which it is shown that two seemingly incompatible types of dynamic behavior are exhibited in the cubic Szeg\H{o} equation: (1) complete integrability, and (2) the existence of arbitrarily large (in the $H^s$ sense) trajectories.

\begin{remark}
In the course of proving Theorem \ref{thm.main}, one may observe that the lower bound on the regularity, $N$, can be improved and better understood by the fact that for $\theta$ belonging to a full measure subset of $\Theta_1$, $N_{\theta}$ can be replaced with $d$. Moreover, on that same subset of $\Theta_1$, with a more careful analysis in Lemma \ref{lem 7.3}, $N$ can be set to be arbitrarily close to $\frac{5}{2}d$. We refer to Remark \ref{rmk.5/2}, for more discussion.
\end{remark}

Let us mention two modify scattering works that are related to this manuscript, \cite{GPT16, L19}. It is useful to compare the work presented in this manuscript to that which is contained in the paper of Grebert, Patural, and Thomann \cite{GPT16}. In  their manuscript, they consider a cubic NLS system on the waveguide manifold, $\mathbb{R}\times \mathbb{T}^d$ for $d=1,...,4$, where the Laplacian is perturbed by a convolution potential. The perturbations of the frequencies are of the form $\lambda_p = \|p\|^2+O(\|p\|^{-m})$ whereas the perturbations presented here are of the form $\lambda_p = \|p\|^2+ O(\|p\|^2)$. Both settings present their own difficulties, but the $O(\|p\|^{-m})$ condition leads a greater ability to control loss of regularity (see Assumption 1.1 from \cite{GPT16}). 

The last work on modified scattering in the waveguide setting that we would like to emphasize is the work of Liu, \cite{L19}, in which the author establishes a modified scattering property of solutions to cubic NLS on the waveguide, $\mathbb{R}\times \mathbb{T}^d$ for $d=1,...,4$, with respect to a different class of asymptotic trajectories than those used in \cite{HPTV15} and in this manuscript.

\subsection{Outline of the Proof of Theorem \ref{thm.main}}

The three parts of Theorem \ref{thm.main} are found throughout this paper as various theorems and remarks. Part (1) of the theorem follows directly from the statement of Theorem \ref{thm.asym} and Remark \ref{rmk.eta}. Part (2) follows from Theorem \ref{thm.asy}. Finally, Part (3) follows from the definition of $\Theta_1$, the definition of $\Theta_2$, and Remark \ref{rmk.Besi}.

The proof of Theorem \ref{thm.asym} makes the bulk of the complexity of the functional setting necessary. The main parts of the proof consists of a small-data wellposedness result (Proposition \ref{prop 6.2}) as well as an asymptotic stability estimate (Lemma \ref{asymptotic}). The strategy of the arguments that justify the main parts of Theorem \ref{thm.asym} are the same as what appears in \cite{HPTV15}. However, the proofs are made more complicated by two factors: (1) the appearance arbitrarily small ``small divisors" in proving Proposition \ref{prop 3.1} and (2) the use of a quasi-resonant effective equation as opposed to a fully resonant effective equation. The small divisor problem forces us to lose some control over the estimates of the nonresonant parts of the cubic nonlinearity (See estimate of $\mathcal{E}^t_{3, M}$ in Proposition \ref{prop 3.1}). This issue requires one to treat the proof of Proposition \ref{prop 6.2} much more delicately. The quasi-resonant effective equations are limited in their effectiveness over arbitrarily long time scales. This requires one to periodically change the effective equation as time approaches infinity. This also affects the way in which one approximates the solution to the full equation by solutions to the effective equation: An arbitrarily small level of additional regularity is necessary to characterize the behavior of solutions to the full equation using uniform approximations of the full equation. We note that this loss of regularity is not completely out of step with the previous work of \cite{HPTV15} and \cite{GPT16} in which additional regularity is needed. However the additional regularity in the previous works (which is also required in this work) is needed only for the $\mathbb{R}$-variable. Finally, the most fundamental difference in this work is the aspect of the dimension of the $d$-dimensional torus. The quasi-resonant structure of the effective system prevents a direct application of the Strichartz estimates used in \cite{HPTV15} (Lemma 7.1) and most importantly \cite{GPT16} (Lemma 3.2). This means that our base norm (the $Z_d$-norm, see Definition \ref{defn Z_d}) must consist of an $\ell^2$ Sobolev norm with more than $d/2$ derivatives. This is overcome by the symmetry that the quasi-resonant system displays at low modes.

The proof of Theorem \ref{thm.asy}, consists of properly reformulating the result of Giuliani and Guardia, \cite{GG22}, to construct a growing solution to a resonant waveguide system (Corollary \ref{GrowCor}), then using the stability estimate referenced in the previous paragraph (Lemma \ref{asymptotic}) to construct a solution to the full equation.

\subsection{Outline of the Rest of This Paper}

The organization of the rest of this paper is: in Section \ref{sec.preliminary}, we present the preliminaries and define the norms that will be used in later sections; in Section \ref{sec.decomp}, we decompose the nonlinearity in the gauge transformed NLS and study the properties of each component in this decomposition; in Section \ref{sec.dynamics} we discuss the dynamic of the quasi-resonant equations; in Section \ref{section.asymptotic} we prove the asymptotic behavior of \eqref{eq NLS}. Also in Appendix \ref{app.Proof3.1} and Appendix \ref{app.nocasc}, we include the proofs of two propositions used in the main theorem.

\subsection*{Acknowledgement}

Part of this work was done while the second author was in residence at the Institute for Computational and Experimental Research in Mathematics (ICERM) in Providence, RI, during the Hamiltonian Methods in Dispersive and Wave Evolution Equations program. The authors would like to thank Benoit Pausader whose mini-course on Semilinear Dispersive Equations at ICERM in 2021 inspired the direction of this project. B. W. is supported by NSF grant DMS 1856124.  X.Y. is partially supported by an AMS-Simons travel grant.

\section{Preliminaries}\label{sec.preliminary}

In this section, we present the harmonic tools, discuss resonance level sets, and  define the norms that will be used in the rest of this paper.

\subsection{Notations}
We use the usual notation that $A \lesssim  B$ or $B \gtrsim A$ to denote an estimate of the form $A \leq C B$, for some constant $0 < C < \infty$ depending only on the {\it a priori} fixed constants of the problem. We also use $a+$ and $a-$ to denote expressions of the form $a + \varepsilon$ and $a - \varepsilon$, for any $0 <  \varepsilon \ll 1$.

\subsection{Fourier Transforms and Littlewood–Paley Projections}
We will consider functions $f : \R \to \C$ and functions $F : \R \times \T_{\theta}^d \to \C$.  To distinguish between them, we use the convention that lower-case letters denote functions defined on $\R$, upper-case letters denote functions defined on $\R \times \T_{\theta}^d$, and calligraphic letters denote operators, except for the Littlewood–Paley operators and dyadic numbers, which are capitalized most of the time.

\subsubsection{Fourier Transforms}
We define the Fourier transform on $\R$ by
\begin{align}
\widehat{g} (\xi) : = \tfrac{1}{2\pi} \int_{\R} e^{-ix \cdot \xi } g(x) \, dx .
\end{align}
If $F(x,y)$ depends on $(x,y) \in \R \times \T_{\theta}^d$, $\widehat{F} (\xi ,y)$ denotes the partial Fourier transform in $x$. 
We also consider the Fourier transform of $f : \T_{\theta}^d \to \C$,
\begin{align}
f_p : = \frac{1}{|\T_{\theta}^d|} \int_{\T_{\theta}^d} f(y) e^{-i \inner{p,y}_{\theta}} \, dy   , \quad p \in \Z^d ,
\end{align}
and this extends to $F(x,y)$. Here $\inner{p,y}_{\theta}:= \sum_{i=1}^d \theta_ip_i y_i$. Finally, we also have the full spatial Fourier transform
\begin{align}
(\mathcal{F} F ) (\xi , p) := \frac{1}{|\T_{\theta}^d|} \int_{\T_{\theta}^d} \widehat{F} (\xi , y) e^{-i \inner{p,y}_{\theta}} \, dy   = \widehat{F}_p (\xi) .
\end{align}

\subsubsection{Littlewood–Paley Projections}
Let us know define Littlewood–Paley projections. For the full frequency space, these are defined as follows: 
\begin{align}
(\mathcal{F} P_{\leq N} ) F (\xi , p) =  \phi (\frac{\xi}{N}) \phi (\frac{p_1}{N}) \cdots \phi (\frac{p_d}{N}) (\mathcal{F} F) (\xi , p),
\end{align}
where $\phi \in C_c^{\infty} (\R)$, $\phi (x) =1 $ when $\abs{x} \leq 1 $ and $\phi (x) =0$ when $\abs{x} \geq 2$. We define
\begin{align}
P_N : = P_{\leq N} - P_{\leq \frac{N}{2}} , \quad P_{\geq N} : = 1 - P_{\leq \frac{N}{2}} . 
\end{align}

We also define Littlewood–Paley projections concentrating on the frequency in $x$ only,
\begin{align}
(\mathcal{F} Q_{\leq N} F) (\xi , p) = \phi (\frac{\xi}{N}) (\mathcal{F} F ) (\xi, p ) , 
\end{align}
with
\begin{align}
Q_N := Q_{\leq N} - Q_{\leq \frac{N}{2}} , \quad Q_{\geq N} := 1 - Q_{\leq \frac{N}{2}} ;
\end{align}
and Littlewood–Paley projections concentrating on the frequency in $y$ only,
\begin{align}
(\mathcal{F} L_{\leq N} F) (\xi , p) = \phi (\frac{p_1}{N}) \cdots \phi (\frac{p_d}{N}) (\mathcal{F} F ) (\xi, p ) , 
\end{align}
with
\begin{align}
L_N := L_{\leq N} - L_{\leq \frac{N}{2}} , \quad L_{\geq N} := 1 - L_{\leq \frac{N}{2}} . 
\end{align}

\subsection{Resonance Level Sets}

We start by establishing notation for the norms of the frequency modes. In the square case, the eigenvalue of $-\Delta_{\mathbb{T}^d}$ associated to the mode $p \in \mathbb{Z}^d$ is simply the Euclidean norm of $p$, which we will denote $\|p\|$. In the case that the torus is re-scaled by a vector, $\theta$, the eigenvalue of $-\Delta_{\mathbb{T}^d_{\theta}}$ associated to the mode $p \in \mathbb{Z}^d$ is a re-scaled version of the Euclidean. We suppress the dependence on $\theta$ and denote this eigenvalue by
    \begin{align*}
        \lambda_p:= \sum \theta_i^2p_i^2.    
    \end{align*}
    We define the following sets corresponding to momentum and resonance level sets:
\begin{align}\label{eq MW}
\begin{aligned}
\mathcal{M} & : = \{ (p,q,r,s) \in (\Z^{d})^4 : p-q + r-s =0\} ,\\
\Gamma_{\omega} & : = \{ (p,q,r,s) \in \mathcal{M} : \lambda_p - \lambda_q + \lambda_r - \lambda_s = \omega \} ,\\
\Omega&:= \{ \omega \in \mathbb{R} : \Gamma_{\omega} \neq \emptyset  \} .
\end{aligned}
\end{align}
Note that $\Omega$ is countable.

\subsection{Functional Setting}
Fix $\delta < 10^{-3}$. $N$ will be specified later. It will depend on $\theta$ and $d$. The requirements that will be given to $N$ will not be designed to be sharp (see the proof of Lemma \ref{lem 7.3} for detailed discussion on the sharpness of the requirement on $N$).

\begin{defn}[$h^k$-norm]
\begin{align}
\norm{\{ z_p\}}_{h_p^k}^2 := \sum_{p \in \Z^d} \left[\sum_{i=1}^d (1+|p_i|^2)^k \right] \abs{z_p}^2 .
\end{align}
\end{defn}

\begin{defn}[$k$-norm]\label{defn k}
Define an auxiliary norm:
	\begin{align*}
		|z|^2_k&:=  \int_{\xi\in \mathbb{R}}\sum_{p  \in \mathbb{Z}^d } \left[\sum_{i=1}^d (1+|p_i|^2)^k + (1+|\xi|^2)^k \right]|z_{\xi, p}|^2 .
	\end{align*}
\end{defn}

\begin{remark}
We note that since 
 	\begin{align*}
 		C_k^{-1} (1+|(\xi, p)|)^{2k} \leq \sum_{i=1}^d(1+|p_i|^2)^k  + (1+|\xi|^2)^k]\leq C_k  (1+|(\xi, p)|)^{2k},
	\end{align*}
	we can conclude that the norms are equivalent:
 	\begin{align*}
 		C^{-1}_k \|z\|^2_{H_{x, y}^k} \leq |z|^2_k \leq C_k \|z\|^2_{H_{x, y}^k} .
 	\end{align*}
     
\end{remark}    

\begin{remark}
We need a norm that plays the same role as the $Z$-norm in the paper of \cite{HPTV15} but controls a higher regularity (depending on the dimension) than the $h^1$ norm that the $Z$-norm controls. We note that this then requires and argument that allows one to bound the higher regularity norm in the same fashion that \cite{HPTV15} controls the $h^1$ norm. It is very important to point out that the $Z_d$-norm is defined to exploit the fact that  $h_p^{\frac{d}{2}+}$ is an algebra.
\end{remark}
\begin{defn}[$Z_d$-norm]\label{defn Z_d}
Define a new $Z_d$ norm
\begin{align}
\norm{F}_{Z_d}^2 : = \sup_{\xi \in \R} (1 + \abs{\xi}^2)^2 \norm{\widehat{F}_p (\xi)}_{h_p^{\frac{d}{2}+}}^2 .
\end{align}
\end{defn}

\begin{defn}[$S$- and $S^+$-norms]
\begin{align}
\norm{F}_{S} & : = \norm{F}_{H_{x,y}^N} + \norm{xF}_{L_x^2} ,\\
\norm{F}_{S^+} & : = \norm{F}_{S} + \norm{(1- \partial_{xx})^4 F}_{S} + \norm{xF}_{S}  .
\end{align}
\end{defn}
Then, since $N \geq 10 d$, we have
\begin{align}\label{eq 14}
\norm{F}_{H_{x,y}^1} \lesssim \norm{F}_{Z_d} \lesssim \norm{F}_{S} \lesssim \norm{F}_{S^+} .
\end{align}

\begin{defn}[$Z_{t,d}$-norm]\label{defn Z_td}
\begin{align}
\norm{F}_{Z_{t,d}} : =\norm{F}_{Z_d} + (1 + \abs{t})^{-\frac{1}{20}} \norm{F}_{S} .
\end{align}
\end{defn}

\begin{remark}
The following two space-time norms are necessary for the well-posedness and stability results of Section \ref{section.asymptotic}. We use the same framework as \cite{HPTV15} for the establishment of the results found in Section \ref{section.asymptotic}.
\end{remark}

\begin{defn}[$X_T$- and $X_T^+$-norms]\label{defn XT}
\begin{align}
\norm{F}_{X_T} & : = \sup_{0 \leq t \leq T} \{ \norm{F(t)}_{Z_d} + (1 + \abs{t})^{-\delta} \norm{F(t)}_{S} + (1+ \abs{t})^{1-3\delta}  \norm{\partial_t F (t)}_{S} \} ,\\
\norm{F}_{X_T^+} & : = \norm{F}_{X_T} + \sup_{0 \leq t \leq T} \{  (1 + \abs{t})^{-5\delta} \norm{F(t)}_{S^+} + (1+ \abs{t})^{1-7\delta}  \norm{\partial_t F (t)}_{S^+} \}  .
\end{align}
\end{defn}

\begin{defn}[$Y$-norm]
Define
\begin{align}
\norm{f}_{Y} : = \norm{\inner{x}^{\frac{9}{10}} f}_{L_x^2} + \norm{f}_{H_x^{\frac{3N}{4}}} .
\end{align}
Then
\begin{align}\label{defn Y}
\sum_{p \in \Z^d} \norm{F_p}_{Y} \lesssim \norm{F}_{S} .
\end{align}
\end{defn}

\begin{remark}

Furthermore, note that the following two basic inequalities hold:
    \begin{align*}
        \|f\|_{L^1(\mathbb{R})}  \lesssim \|f\|_{L^2(\mathbb{R})}^{\frac{1}{2}}\|xf\|_{L^2(\mathbb{R})}^{\frac{1}{2}} ,
    \end{align*}
and
    \begin{align*}
        \|f\|_{H^k(\mathbb{R} \times \mathbb{T}^d)} \lesssim \|f\|_{L^2(\mathbb{R} \times \mathbb{T}^d)}^{\frac{1}{2}}\|f\|_{H^{2k}(\mathbb{R} \times \mathbb{T}^d)}^{\frac{1}{2}} .
    \end{align*}
These two inequalities allow us to use a similar argument to that which appears in \cite{HPTV15} to conclude that 
    \begin{align*}
        \|F\|_{Z_d} \lesssim \|F\|_{L^2_{x, y}}^{\frac{1}{4}}\|F\|_{S}^{\frac{3}{4}} .
    \end{align*}
It is important to observe that this inequality requires that $N> d$ in order to hold.
\end{remark}

\begin{remark}
Let us present the last elementary inequality in this section, which will be used frequently when summing the frequencies in later sections. 
\begin{align}\label{eq Young}
\norm{\sum_{(q,r,s): (p,q,r,s) \in \mathcal{M}} c_q^{(1)} c_r^{(2)} c_s^{(3)}}_{\ell_p^2} \lesssim \min_{\sigma \in S_3} \norm{c^{\sigma(1)}}_{\ell_p^2} \norm{c^{\sigma(2)}}_{\ell_p^1} \norm{c^{\sigma(3)}}_{\ell_p^1} .
\end{align}
\end{remark}

We will close this section with the following definition of a cutoff function
\begin{defn}[A cutoff function]\label{defn cutoff}
For $T \gtrsim 1$ a positive number, we let $q_T : \R \to \R$ be an arbitrary function satisfying
\begin{align}
& 0 \leq q_T (s) \leq 1, \\
& q_T (s) =0 \quad \text{if } \abs{s}\leq \tfrac{T}{4} \text{ or } \abs{s} \geq T,\\
& \int_{\R} \abs{q_T' (s)} \, ds \leq 10 .
\end{align}
Particular examples are the characteristic functions $q_T (s) = \mathbf{1}_{[\frac{T}{2} ,T]} (s)$, with the natural interpolation of the integral on $\R$ of $\abs{q_T'}$. 
\end{defn}

\section{Analysis of the Nonlinearity}\label{sec.decomp}

In this section, we provide a decomposition of the nonlinearity in the gauge transformed equation \eqref{eq newNLS}, and prove  its decay properties in Proposition \ref{prop 3.1}.

\subsection{Duhamel's Formula and Gauge Transformation}

Define the standard gauge transformation by 
\begin{align}\label{eq U}
U(t,x,y) =  e^{it \Delta_{\R \times \T_{\theta}^d}} F(t) = \sum_{p \in \Z^d} e^{i \inner{p,y}_{\theta}} e^{-it \lambda_p} (e^{it\partial_{xx}} F_p (t)) (x) . 
\end{align}
 Then we see that  $U$ solves \eqref{eq NLS} if and only if $F$ solves
\begin{align}\label{eq newNLS}
i \partial_t F (t) = e^{-it \Delta_{\R \times \T_{\theta}^d}} ( e^{it \Delta_{\R \times \T_{\theta}^d}}  F(t) \cdot  e^{-it \Delta_{\R \times \T_{\theta}^d}} \overline{F(t)} \cdot  e^{it \Delta_{\R \times \T_{\theta}^d}} F(t)) = : \mathcal{N}^t [F(t) , F(t) , F(t)]. 
\end{align}
Denote the nonlinearity in \eqref{eq newNLS} by $\mathcal{N}^t [F(t) , F(t) , F(t)]$, where  the trilinear form is defined as follows
\begin{align}
\mathcal{N}^t [F,G,H] : = e^{-it \Delta_{\R \times \T_{\theta}^d}} (e^{it \Delta_{\R \times \T_{\theta}^d}} F \cdot e^{-it \Delta_{\R \times \T_{\theta}^d}}  \overline{G} \cdot e^{it \Delta_{\R \times \T_{\theta}^d}} H)
. 
\end{align}

\begin{remark}[Properties of $\mathcal{N}^t$] 
We list some properties of $\mathcal{N}^t$ that will be needed in later sections.
\begin{enumerate}
\item {\bf Fourier transforms of  $\mathcal{N}^t$:}

\noindent The Fourier transform of $\mathcal{N}^t$ is given by
\begin{align}\label{eq N}
\mathcal{FN}^t [F, G, H] (\xi, p) & = \sum_{(p,q,r,s) \in \mathcal{M}} e^{it (\lambda_p - \lambda_q + \lambda_r - \lambda_s)} (\mathcal{I}^t [F_q, G_r , H_s])^{\wedge} (\xi) \\
& = \sum_{\omega \in \Omega}\sum_{(p,q,r,s) \in \mathcal{M}} e^{it \omega} (\mathcal{I}^t [F_q, G_r , H_s])^{\wedge} (\xi)
\end{align}
where
\begin{align}
\mathcal{I}^t [f,g,h] & :  = e^{-it \partial_{xx}} (e^{it \partial_{xx}} f \cdot e^{-it \partial_{xx}} \overline{g} \cdot e^{it \partial_{xx}}  h) .
\end{align}

\item  {\bf Re-writing $\mathcal{N}^t$:}

\noindent Let $\mathcal{U} (t) =e^{it \partial_{xx}} $. We can write
\begin{align}
\mathcal{I}^t [f,g,h] & : = \mathcal{U} (-t) (\mathcal{U} (t) f  \cdot  \overline{\mathcal{U } (t)g} \cdot  \mathcal{U} (t) h) , 
\end{align}
also the  Fourier transform of $\mathcal{I}^t$  is given by
\begin{align}\label{eq It}
(\mathcal{I}^t [f,g, h]  )^{\wedge} (\xi)= \int_{\R^2} e^{2 it  \eta \kappa} \widehat{f} (\xi - \eta) \overline{\widehat{g}} (\xi - \eta - \kappa) \widehat{h} (\xi - \kappa) \, d\kappa d \eta .
\end{align}
Using \eqref{eq It}, we can also write 
\begin{align}
\mathcal{FN}^t [F, G, H] (\xi, p) = \sum_{\omega \in \Omega}\sum_{(p,q,r,s) \in \mathcal{M}} e^{it \omega}   \int_{\R^2} e^{it 2 \eta \kappa} \widehat{F_q} (\xi - \eta) \overline{\widehat{G_r}} (\xi - \eta - \kappa) \widehat{H_s} (\xi - \kappa) \, d\kappa d \eta .
\end{align}

\item {\bf Leibniz rules for $\mathcal{I}^t$ and $\mathcal{N}^t$: }

\noindent A Leibniz rule for $\mathcal{I}^t [ f, g, h]$
\begin{align}
Z \mathcal{I}^t [f,g,h] =   \mathcal{I}^t [Zf,g,h] + \mathcal{I}^t [f,Zg,h] + \mathcal{I}^t [f,g,Zh], \quad Z \in \{ ix, \partial_x \} . 
\end{align}
A similar property holds for the whole nonlinearity $\mathcal{N}^ t[F,G, H]$, where $Z$ can also be a derivative in the transverse direction $Z = \partial_{y_j}$.
\end{enumerate}
\end{remark}

\subsection{Decomposition of the Nonlinearity}
We decompose the nonlinearity in equation \eqref{eq newNLS} in the following way:
\begin{align}\label{eq eff+E}
 i \partial_t F = \mathcal{N}^t (F) = \mathcal{N}_{eff} (F) + \mathcal{E}^t (F) ,
\end{align}
where $ \mathcal{E}^t (F) $ is integrable in time. One hopes to prove that the asymptotic dynamics converge to that of the effective system
\begin{align}
i\partial_t G = \mathcal{N}_{eff} (G) .
\end{align}

Let $M\geq 0$. We begin by decomposing $\mathcal{N}^t$ in the following way with respect to $M$:
\begin{align*}
\mathcal{N }^t [F, G,H] & =  \Pi_M^t [F,G,H] + \widetilde{\mathcal{N}}_M^t [F,G,H].
\end{align*}
Here we define $\Pi_M^t$ and $\widetilde{\mathcal{N}}_M^t$ by 
\begin{align}\label{eq Pi}
\begin{aligned}
    \mathcal{F} \Pi_M^t [F, G,H] (\xi , p)&  : = \sum_{\abs{\omega} < \frac{1}{M}} \sum_{(p,q,r,s) \in \Gamma_{\omega}}  e^{it \omega} (\mathcal{I}^t [F_q, G_r , H_s] )^{\wedge}(\xi) \\
& = \sum_{\abs{\omega} < \frac{1}{M}} \sum_{(p,q,r,s) \in \Gamma_{\omega}}  e^{it \omega} \int_{\R^2} e^{it 2 \eta \kappa} \widehat{F_q} (\xi - \eta) \overline{\widehat{G_r}}  (\xi - \eta - \kappa) \widehat{H_s} (\xi - \kappa) \, d\kappa d \eta \\
\mathcal{F} \widetilde{\mathcal{N}}_M^t [F, G,H] (\xi , p) & : = \sum_{\abs{\omega} \geq \frac{1}{M}} \sum_{(p,q,r,s) \in \Gamma_{\omega}}  e^{it \omega} (\mathcal{I}^t [F_q, G_r , H_s] )^{\wedge}(\xi) .
\end{aligned}
\end{align}
Here $\Pi^t_0:= \mathcal{N}^t$ and $\widetilde{\mathcal{N}}^t_0=0$. 
Our goal is to show that as time approaches infinity, $\Pi_M^t$ ``resembles" $\frac{\pi}{t} \mathcal{R}^t_M$ defined by
    \begin{align}
        \mathcal{F} \mathcal{R}^t_M  [F,G,H] (\xi, p) & : = \sum_{\abs{\omega} < \frac{1}{M}} e^{it \omega}\sum_{(p,q,r,s) \in \Gamma_{\omega}} \widehat{F_q} (\xi) \overline{\widehat{G_r}} (\xi) \widehat{H_s} (\xi) \\
         \mathcal{F} \mathcal{R}_0^t  [F,G,H] (\xi, p) & : = \sum_{\omega \in \Omega} e^{it \omega}\sum_{(p,q,r,s) \in \Gamma_{\omega}} \widehat{F_q} (\xi) \overline{\widehat{G_r}} (\xi) \widehat{H_s} (\xi) . 
    \end{align}
Specifically, we want to show that (see the precise estimates in Lemma \ref{lem 3.7})
\begin{align}
\Pi_M^t = \frac{\pi}{t} \mathcal{R}_M^t  + O (\abs{t}^{-1 - 2 \delta}) .
\end{align}
The effective system which one would hope to prove that the asymptotic dynamics converge to is given by
\begin{align}\label{eq Resonance}
i \partial_{t} G(t) =  \mathcal{R}^{t}_M [G(t), G(t) , G(t)]. 
\end{align}

In order to accomplish this, one then needs to show that the affects of the the nonlinearity, $\mathcal{\widetilde{N}}_M^t$, on the dynamics diminish over time.  To deal with $\widetilde{\mathcal{N}}_M^t $, we further decompose $\widetilde{\mathcal{N}}_M^t$ along the nonresonant level sets
\begin{align}
\mathcal{F} \mathcal{\widetilde{N}}_M^t [F, G, H] (\xi, p) & = \sum_{\abs{\omega} \geq \frac{1}{M}} \sum_{(p,q,r,s) \in \Gamma_{\omega}} e^{it \omega} (\mathcal{O}_1^t [F_q, G_r, H_s] (\xi) + \mathcal{O}_2^t [F_q, G_r, H_s] (\xi) ) , 
\end{align}
where
\begin{align}\label{eq O}
\begin{aligned}
\mathcal{O}_1^t [f,g,h] (\xi) & : = \int_{\R^2} e^{2it\eta \kappa} (1 -\phi (t^{\frac{1}{4}} \eta \kappa)) \widehat{f} (\xi -\eta) \overline{\widehat{g}} (\xi - \eta - \kappa) \widehat{h} (\xi - \kappa ) \, d \eta d \kappa ,\\
\mathcal{O}_2^t [f,g,h] (\xi) & : = \int_{\R^2} e^{2it\eta \kappa}  \phi (t^{\frac{1}{4}} \eta \kappa) \widehat{f} (\xi -\eta) \overline{\widehat{g}} (\xi - \eta - \kappa) \widehat{h} (\xi - \kappa ) \, d \eta d \kappa , \\
(\mathcal{I}^t [f,g,h])^{\wedge} (\xi) & = \mathcal{O}_1^t [f,g,h] (\xi) +  \mathcal{O}_2^t [f,g,h] (\xi) .
\end{aligned}
\end{align}
We further define $\mathcal{E}_{1, M}^t$ and $\mathcal{E}_{2, M}^t$:
\begin{align}
\mathcal{F} \mathcal{E}_{1, M}^t [F, G, H] (\xi, p) = \sum_{\abs{\omega} \geq \frac{1}{M}} \sum_{(p,q,r,s) \in \Gamma_{\omega}} e^{it \omega} \mathcal{O}_1^t [F_q, G_r, H_s] (\xi)  \\
\mathcal{F} \mathcal{E}_{2, M}^t [F, G, H] (\xi, p) = \sum_{\abs{\omega} \geq \frac{1}{M}} \sum_{(p,q,r,s) \in \Gamma_{\omega}} e^{it \omega} \mathcal{O}_2^t [F_q, G_r, H_s] (\xi) . 
\end{align}
Note that 
\begin{align}
\mathcal{\widetilde{N}}_M^t=\mathcal{E}_{1, M}^t+\mathcal{E}_{2, M}^t .
\end{align}
Finally we define an auxiliary nonlinearity, $\mathcal{E}^t_{3, M}$, by 
\begin{align}
\mathcal{F} \mathcal{E}_{3,M}^t (\xi , p) : = \sum_{\abs{\omega} \geq \frac{1}{M}} \sum_{(p,q,r,s) \in \Gamma_{\omega}} \frac{e^{it\omega}}{i\omega} \mathcal{O}_2^t [F_q, G_r, H_s] (\xi) . 
\end{align}
The following estimate pertaining to the relationship between $\mathcal{E}^t_{3, M}$ and $\mathcal{E}^t_{2, M}$ is addressed in Remark \ref{remark.Eerror}:
    \begin{align*}
        \mathcal{E}_{2,M} (t) = \partial_t \mathcal{E}_{3,M} (t) +O(t^{-(1+\delta)})
    \end{align*}
    in the $Z_d$ norm for $t \geq 1$. The error between $ \mathcal{E}_{2,M} (t)$ and $\partial_t \mathcal{E}_{3,M} (t)$ will be denoted, $\mathcal{E}^t_{err, M}$, and defined in \eqref{eq 22}.

Now we have the following full decomposition for $\mathcal{N}^t$
    \begin{align}\label{eq Decomp}
        \mathcal{F} N^{t}=\Pi^{t}_M+ \underbrace{\mathcal{E}_{1, M}^{t}   + \mathcal{E}_{2, M}^{t} }_{\widetilde{\mathcal{N}}_M^t}=  (\Pi^{t}_M - \frac{\pi}{t}\mathcal{R}_M^{t})+ \mathcal{E}_{1, M}^{t}   + \mathcal{E}_{err, M}^{t}+ \partial_{t} \mathcal{E}_{3,M}^{t}+\frac{\pi}{t} \mathcal{R}_M^{t}.
    \end{align}

\subsection{Main Proposition}
After decomposing the nonlinearity in \eqref{eq newNLS}, we state the following main proposition that is an adapted version of Proposition 3.1 in \cite{HPTV15}. It is worth pointing out that a major difference is in \eqref{eq E3M}, where the upper bound depends on $M$ (while in \cite{HPTV15} it is not).

\begin{proposition}\label{prop 3.1}
For $T \geq 1$, assume that $F, G,H: \R \to S$ satisfy 
\begin{align}\label{eq 3.2}
\norm{F}_{X_T} + \norm{G}_{X_T} + \norm{H}_{X_T} \leq 1  . 
\end{align}
Then, for $t \in [\frac{T}{4} , T]$, if we let
\begin{align}
 \mathcal{N}^t-\frac{\pi}{t}\mathcal{R}_M^t=\mathcal{N}^t[F(t), G(t) , H(t)] - \frac{\pi}{t}\mathcal{R}_M^t[F(t), G(t) , H(t)] ,
\end{align}
then the following bounds hold uniformly in $T \geq 1$, $M\geq 0$:
\begin{align}
T^{- \delta} \norm{\int_{\R}  q_T (t) \cdot(\mathcal{N}^t-\frac{\pi}{t}\mathcal{R}_M^t) \, dt }_{S} \lesssim 1 , \\
T^{1+ \delta} \sup_{\frac{T}{4} \leq t \leq T} \norm{\mathcal{E}_{1,M}^t+\mathcal{E}_{err, M}^{t}}_{Z_d}  \lesssim 1 , \\
T^{\frac{1}{10}} \sup_{\frac{T}{4} \leq t \leq T} \norm{\mathcal{E}_{3,M}^t}_{S} \lesssim M  \label{eq E3M},
\end{align}
where the function $q_T(t)$ is defined in Definition \ref{defn cutoff}.

Assuming in addition that 
\begin{align}\label{eq 3.4}
\norm{F}_{X_T^+} + \norm{G}_{X_T^+} + \norm{H}_{X_T^+} \leq 1,
\end{align}
we also have that
\begin{align}
T^{-5\delta} \norm{\int_{\R} q_T (t) \cdot(\mathcal{N}^t-\frac{\pi}{t}\mathcal{R}_M^t) \, dt}_{S^+ } + T^{2\delta} \norm{\int_{\R} q_T(t) \cdot(\mathcal{N}^t-\frac{\pi}{t}\mathcal{R}_M^t) \, dt}_{S} \lesssim 1.
\end{align}
\end{proposition}

The proof of Proposition \ref{prop 3.1} can be found in Appendix \ref{app.Proof3.1}. Since the techniques and lemmas used in this proof are very similar to what appears in the work of Hani, Pausader, Tzvetkov, and Visciglia \cite{HPTV15}, those details are relegated to the Appendix \ref{app.Proof3.1}. There are important technicalities to verify in this new setting as well as some small changes to the proof strategy and functional setting. However, these changes do not present difficulties significant enough to include in the main body of this manuscript. We include the work in the appendix for completeness. Finally, we remark that if one replaces the $\norm{\cdot}_{X_T} \leq 1$  condition with $\norm{\cdot}_{X_T} \leq \varepsilon$ (or $\norm{\cdot}_{X_T^+} \leq 1$ by $\norm{\cdot}_{X_T^+} \leq \varepsilon$), then since the operators are trilinear, the bounds gain a factor of $\varepsilon^3$.

\section{Dynamics of the Quasi-Resonant Equations}\label{sec.dynamics}

In the following section, we begin the study of the effects of resonance and quasi-resonance on the dynamics of the full nonlinear Schr\"odinger system.

        \subsection{Irrationality}\label{subsec.irrationality}

The finer arithmetic properties of the irrational vector $\theta= (\theta_1, ..., \theta_d)$ will play a crucial role in dynamics of the quasi-resonant equations associated to the cubic NLS. We will define a full measure subset of  $\mathbb{R}_+^d$, $\Theta_1$, to which we will restrict our re-scalings of the $d$-dimensional torus. The definition of the set constructed will be informed by the works of \cite{GG22} and  \cite{SW20}.

We begin with a proposition on the size of the set of Diophantine vectors in $\mathbb{R}^d_+$.

    \begin{proposition}
        The subset of $\mathbb{R}^d_+$,
        \begin{align*}
            \left\{ \theta \in \mathbb{R}^d_+~:~ \exists C_{\theta}>0 \mbox{ such that }\left| \sum_{i=1}^d \theta_i^2 n_i \right| \geq C_{\theta}\|(n_1, ..., n_d)\|^{-d} \mbox{ for all }   (n_1, ..., n_d) \in \mathbb{Z}^d\setminus \{0\}\right\}
        \end{align*}
        is of full measure.
    \end{proposition}
    See \cite{C57} and \cite{DGG17} for more information. Moreover, we have another set of full measure to consider
    \begin{proposition}\label{Khintchin}
        The subset of $\mathbb{R}^d_+$,
        \begin{align*}
            \left\{ \theta \in \mathbb{R}^d_+~:~ \left| \theta_1^2 n_1+ \theta_2^2 n_2 \right| \leq \|(n_1, n_2)\|^{-1} \mbox{ for all } (n_1, n_2) \in \mathbb{Z}^2\setminus \{0\}  \right\}
        \end{align*}
        is of full measure.
    \end{proposition}
    
   Proposition \ref{Khintchin} is a corollary to the theorem of Khintchine,  \cite{K24}, that appears in \cite{GG22} as Theorem 2.5. It is important to observe that the set of vectors in Proposition \ref{Khintchin} only imposes conditions on the first two components of the vectors $\theta$. In the paper of Giuliani and Guardia \cite{GG22}, solutions for a two-dimensional torus system are constructed. There is a simple method to extend those solutions to arbitrarily high-dimensional solutions. This is the method that we will employ, so it suffices to first construct solutions on the first two components of the $d$-dimensional irrational torus. 
   
   One may observe that a direct application of the theorems pertaining to the measures of the subsets described above would yield a description of vectors defined with the squares of the components $(\theta^2_1, ..., \theta^2_d)$ instead of the vectors $\theta$ since the conditions are defined with respect to the squares of the components. This is overcome by observing that the map, $\theta \mapsto (\theta^2_1, ..., \theta^2_d)$, is locally bi-Lipschitz and thus preserves Lebesgue nullity of measures of sets and Hausdorff dimension.
   
   Finally, we define the generic set of irrational vectors that we are interested in considering:
   \begin{defn}\label{def.Theta}
        Let $\Theta_2 \subset \Theta_1 \subset \mathbb{R}_+^d$ be defined by 
            \begin{align*}
                \Theta_1&:= \bigcup_{\gamma \geq d} \left\{ \theta \in \mathbb{R}^d_+~:~ \exists C_{\theta}>0 \mbox{ such that }\left| \sum_{i=1}^d \theta_i^2 n_i \right| \geq C_{\theta}\|(n_1, ..., n_d)\|^{-\gamma} \mbox{ for all }   (n_1, ..., n_d) \in \mathbb{Z}^d\setminus \{0\}\right\} , \\
                \Theta_2&:= \Theta_1 \cap \left\{ \theta \in \mathbb{R}^d_+~:~ \left| \theta_1^2 n_1+ \theta_2^2 n_2 \right| \leq \left( \|(n_1, n_2)\|\log \|(n_1, n_2)\|\right)^{-1} \mbox{ for all } (n_1, n_2) \in \mathbb{Z}^2\setminus \{0\}  \right\} .
            \end{align*}
   \end{defn}

\begin{remark}\label{rmk.Besi}
   Finally, the following theorem of Jarnik and Besicovitch, \cite{B34}, implies that the complement of $\Theta_1$ has Hausdorff dimension zero.
   
   \begin{theorem}
        The set 
        \begin{align}
            E_q:= \left\{x\in \mathbb{R}~:~ |x-\frac{m}{n}| \leq \frac{1}{n^q} \mbox{ for infinitely many rationals } \frac{m}{n} \right\} 
        \end{align}
        has Hausdorff dimension $2/q$.
   \end{theorem}
\end{remark}

\subsection{Irrational Resonances and Quasi-Resonances in $\mathbb{Z}^d$}\label{subsec.resonances}

We consider the Hamiltonian, $\widetilde{h}$, associated to the cubic NLS on the $d$-dimensional torus \eqref{eq TNLS}:
	\begin{align} \label{compNLSHam}
		\widetilde{h}(\{z_p\})&= h_0 + N \nonumber\\
				&= \tfrac{1}{2}\sum_{p \in \mathbb{Z}^d} \lambda_p |z_p|^2+ \tfrac{1}{4} \sum_{p+r=q+s} z_{q}z_{s}\overline{z}_{r}\overline{z}_{p}.
	\end{align}

We recall that
	\begin{align*}
		\mathcal{M} & : = \{ (p,q,r,s) \in (\Z^{d})^4 : p-q + r-s =0\} ,\\
\Gamma_{0} & : = \{ (p,q,r,s) \in \mathcal{M} : \lambda_p - \lambda_q + \lambda_r - \lambda_s = 0 \},
	\end{align*}
where $\Gamma_0$ will denote the set of resonances whose definition appears in \eqref{eq MW}. 

For $K>0$, let  $\Gamma_0^K$ will denote the set of resonances below a fixed frequency level
	\begin{align*}
		\Gamma_0^K&:= \left\{ (p, q, r, s) \in \Gamma_0~:~ |p|, |q|, |r|, |s|\leq K\right\}.
	\end{align*}

For any vector, $\theta= (\theta_1, ..., \theta_d)$, the resonance condition $ \lambda_{p}+\lambda_{r}=\lambda_{q}+\lambda_{s}$ is equivalent to 
		\begin{align*}
			 \sum_{i=1}^d\theta_i^2\left[p_i^2 +r_i^2 -q_i^2 - s_i^2\right]=0.
		\end{align*}
If $\theta$ is an irrational vector, then the above identity is equivalent to the following set of independent one-dimensional resonance equations
		\begin{align*}
			p_i^2 +r_i^2 =q_i^2 + s_i^2 \mbox{ for } i=1, ..., d.
		\end{align*}
	By the definition of vector addition, the conservation of momentum condition also decomposes into one-dimensional conservation of momentum equations.  
	Thus $\Gamma_0= \cap_{i=1}^d\Gamma_{0,i}$, where
		\begin{align*}
			\Gamma_{0,i} := &\left\{ (p, q, r, s) \in (\mathbb{Z}^d)^4~:~ p_i +r_i =q_i + s_i\right\}\bigcap \left\{ (p, q, r, s) \in (\mathbb{Z}^d)^4~:~ p_i^2 +r_i^2 =q_i^2 + s_i^2\right\}.
		\end{align*}

The nice feature about the dimensional decomposition of $\Gamma_0$ is that the one-dimensional resonances are rather simple.  In fact, we can state a simpler characterization with a basic arithmetic lemma:
	\begin{lemma}\label{char}
		For $i=\{1,...,  d\}$,
		\begin{align*}
			\Gamma_{0,i} &= \left\{ (p, q, r, s) \in (\mathbb{Z}^d)^4~:~ p_i=q_i \mbox{ and  } r_i=s_i \right\} \bigcup \left\{ (p, q, r, s) \in (\mathbb{Z}^d)^4~:~ p_i=s_i \mbox{ and } r_i=q_i \right\} .
		\end{align*}
	\end{lemma}
The proof is fairly straightforward and well-known. We refer to \cite{SW20} for the very short proof.

	\subsection{Rational Quasi-Resonances}

	In this subsection , we consider the two-dimensional case because that is the setting in which Giuliani and Guardia \cite{GG22} work. Let $\vec{m}=(m_1, m_2) \in \mathbb{Z}^2_+$ be positive integers. We define the $\vec{m}$- resonances, $\Gamma_{0, \vec{m}} \subset (\mathbb{Z}^d)^4$,  by the following condition: $(p, q, r, s) \in \Gamma_{0, \vec{m}}$ if both of the following identities hold
	\begin{align*}
		p+r&=q+s ,\\
		m_1\left[p_1^2 +r_1^2\right]+m_2\left[p_2^2 +r_2^2\right]&=   m_1\left[q_1^2 +s_1^2\right]+m_2\left[q_2^2 + s_2^2\right] .
	\end{align*}
	
As described above, elements in $\Gamma_{0, \vec{m}}$ can form parallelograms that are not axis-parallel. On the other hand, if $(m_1, m_2)$ is close to $(\theta^2_1, \theta^2_2)$, then elements of  $\Gamma_{0, \vec{m}}$ contain quartets that are quasi-resonant with respect to $(\theta_1, \theta_2)$. We refer the reader to \cite{GG22} for further discussion.

		\subsection{Growth}\label{subsec.growthdynamics}

Consider the following resonant (resonant only in the non-compact variable) Hamiltonian associated to equation \eqref{eq NLS}:
	\begin{align}\label{Rreso}
		H^*(\psi)&= H_0(\psi)+N^*(\psi)\\
				&=\tfrac{1}{2} \int_{\xi \in \mathbb{R}} \sum_{p \in \mathbb{Z}^d} (|\xi|^2+ \lambda_{ p})|\widehat{\psi}(\xi, p)|^2+ \tfrac{1}{4}\int_{\xi \in \mathbb{R}} \sum_{p+r=q+s} \widehat{\psi}(\xi, q) \widehat{\psi}(\xi, s) \overline{\widehat{\psi}}(\xi, r)  \overline{\widehat{\psi}}(\xi, p)\nonumber .
	\end{align} 
The heuristics of the growth argument are as follows. Given a fixed time interval $[0, T]$, we determine a suitable $\vec{m}$ resonant system, 
	\begin{align*}
		H^*_{\vec{m}}(\psi)&
				=\tfrac{1}{2} \int_{\xi \in \mathbb{R}}  \sum_{p \in \mathbb{Z}^d}(|\xi|^2+ \lambda_{ p})|\widehat{\psi}(\xi, p)|^2+ \tfrac{1}{4} \int_{\xi \in \mathbb{R}}\sum_{p+r=q+s \atop (p, q, r, s) \in \Gamma_{\vec{m}}} \widehat{\psi}(\xi, q) \widehat{\psi}(\xi, s) \overline{\widehat{\psi}}(\xi, r) \overline{\widehat{\psi}}(\xi, p).
	\end{align*}
In the the exact same fashion as in Giuliani and Guardia \cite{GG22}, a solution to the equation determined by $H^*_{\vec{m}}$ is constructed to grow over time. More importantly, the flow determined by $H^*_{\vec{m}}$ is conditioned to be a perturbation of the Hamiltonian flow determined by $H^*$ which allows one to construct a solution to the equation determined by $H^*$ that grows on the time interval $[0, T]$.

The fact that the approximation of the dynamics depend on the time interval affects the potential to replicate the argument of \cite{HPTV15} in order to construct solutions whose regularity grows to infinity as time approaches infinity. In their work, a solution to the completely resonant equation is constructed whose $k$-Sobolev norm grows to infinity as time approaches infinity. Modified scattering allows them to relate that solution to a solution to the original cubic NLS system. As will be demonstrated later and which is apparent in the proof of the main theorem of Giuliani and Guardia \cite{GG22}, initial data that is frequency localized will not grow for all $H^*_{\vec{m}}$ systems. In particular, the size of the Fourier support at time zero for a solution to $H^*_{\vec{m}}$ whose $h^k$ norms grow to infinity as time approaches infinity must grow as $(m_1, m_2)$ better approximates $(\theta^2_1, \theta^2_2)$.

Specifically, the authors in Giuliani and Guardia \cite{GG22} prove the following theorem
	\begin{theorem}\label{GrowThm}
		There exists a set of irrational numbers $\mathcal{S} \subset [1, \infty)$, which has full measure and Hausdorff dimension 1, such that for all $\omega \in \mathcal{S}$ the following holds.

	Let $k>0$, $k \neq 1$, and fix $\mathcal{C}\gg 1$. Then the cubic NLS on the irrational torus $\mathbb{T}_{(1, \omega)}^2$ possesses a solution $u(t)$ such that 
		\begin{align*}
			\|u(T)\|_{h^k} \geq \mathcal{C}\|u(0)\|_{h^k},
		\end{align*}
	where 
		\begin{align*}
			T \leq \exp(\mathcal{C}^{\beta})
		\end{align*}
	for some constant $\beta=\beta(k)>0$. Moreover, there exists a constant $\eta= \eta(k)>0$ such that 
		\begin{align*}
			\|u(t)\|_{L^2(\mathbb{T}^2_{(1, \omega)})}\leq \mathcal{C}^{-\eta}
		\end{align*}
	for all $t \in [0, T]$.
	\end{theorem}

The authors further demonstrate the existence of a family of fractal sets for which the rate of growth solutions is increased as the dimension of the set of $\omega$ decreases. For our purposes, we will consider only the generic case. We note that the definition of $\mathcal{S}$ is of full measure by Proposition \ref{Khintchin}. Moreover one can see that the set $ \{(1, \omega) : \omega \in \mathcal{S}\}$ encapsulates our notion (up to rescalings) of good approximation by rationals:
\begin{align}
    \left| \theta_1^2 n_1+ \theta_2^2 n_2 \right| \leq (\|(n_1, n_2)\|\log \|(n_1, n_2)\|)^{-1} \mbox{ for all } (n_1, n_2) \in \mathbb{Z}^2\setminus \{0\} .
\end{align}
Thus if $\theta \in \Theta_2$, then $(1 , \theta_2/\theta_1) \in \mathcal{S}$.

The following simple corollary to Theorem \ref{GrowThm}, given by freezing the noncompact variable and $d-2$ torus variables, then follows:
	\begin{corollary}[Corollary to Theorem \ref{GrowThm}]\label{GrowCor}
		Let $\theta \in \Theta_2$. Let $k>0$, $k \neq 1$, and fix $\mathcal{C}\gg 1$. Then the Hamiltonian system associated to \eqref{Rreso} possesses a solution $v(t)$ such that 
		\begin{align*}
			\|v(T)\|_{H^k} \geq \mathcal{C}\|v(0)\|_{H^k},
		\end{align*}
	where 
		\begin{align*}
			T \leq \exp(\mathcal{C}^{\beta})
		\end{align*}
	for some constant $\beta=\beta(k)>0$. Moreover, there exists a constant $\eta= \eta(k)>0$ such that 
		\begin{align*}
			\|v(t)\|_{L^2_{\xi, p}} \leq \mathcal{C}^{-\eta}
		\end{align*}
	for all $t \in [0, T]$.
	\end{corollary}

	\begin{proof}[Proof of Corollary \ref{GrowCor}]\
	We first observe that with $a_{\xi}(p):=\widehat{v}(\xi, p)$, if $v$ is a solution to  the Hamiltonian system associated to \eqref{Rreso}, then $a$ satisfies the following equation:
		\begin{align*}
			i\partial_t a_{\xi}(p)= (|\xi|^2+ \lambda_{ p} )a_{\xi}(p) - \sum_{p=q+s-r} a_{\xi}( q) a_{\xi}( s) \overline{a}_{\xi}( r) .
		\end{align*}
The gauge transformation, $b_{\xi}(p):= e^{i|\xi|^2t}a_{\xi}(p)$, reduces the above equation to 
 		\begin{align*}
			i\partial_t b_{\xi}(p)= \lambda_{p} b_{\xi}(p) - \sum_{p=q+s-r} b_{\xi}( q) b_{\xi}( s) \overline{b}_{\xi}( r) .
		\end{align*}

	Let $u(\cdot): [0, \infty)  \to H^k(\mathbb{T}^2_{(\theta_1, \theta_2)})$ be a solution to the two-dimensional cubic NLS equation on the torus, $\mathbb{T}^2_{(\theta_1, \theta_2)}$, given by Theorem \ref{GrowThm} such that 
		\begin{align}
		\|u(t)\|_{L^2(\mathbb{T}^2_{(\theta_1, \theta_2)})} \leq \mathcal{C}^{-\eta}.
		\end{align}
		Then $w(t, x_1, ..., x_d):=u(t, x_1, x_2)$ solves the cubic NLS equation on $\mathbb{T}^d_{\theta}$ \eqref{eq TNLS} and $\|w\|_{h^k(\mathbb{T}^d_{\theta})} = \|v\|_{h^k(\mathbb{T}^2_{(\theta_1, \theta_2)})}$ for all $k\geq 0$.  Let $\varphi \in C^{\infty}_c(\mathbb{R})$ be a nonnegative, smooth bump satisfying $\varphi(\xi)= 1$ for $|\xi| \leq 1$ and $\varphi(\xi)=0$ for $|\xi|>2$. Now let
		\begin{align*}
			b_{\xi}(0, p):= \widehat{w}(0, p)\varphi(\xi).
		\end{align*}
Then the following identity holds:\footnote{This is the same computation discussed in Remark 4.2 of \cite{HPTV15}.}
		\begin{align*}
			b_{\xi}(t, p)= \widehat{w}(\varphi(\xi)^2t, p)\varphi(\xi) .
		\end{align*}
Finally, $\|v(t)\|_{H^k} \sim \|w(t)\|_{h^k}=\|u(t)\|_{h^k}$ for all $t \in [0, T]$ and $k \in [0, \infty)$ and the result follows from the estimates of Theorem \ref{GrowThm}.
\end{proof}

\subsection{Boundedness}\label{subsec.boundeddynamics}
Let $\theta \in \Theta_1$ as in Definition \ref{def.Theta}.

	For $M>0$ consider the following Hamiltonian related to the height-$1/M$ quasi-resonant nonlinearity.
	\begin{align}\label{eq ResHam}
		H_M(\psi)&= H_0(\psi)+N_M(\psi)\\
				&=\tfrac{1}{2} \int_{\xi \in \mathbb{R}}  \sum_{p \in \mathbb{Z}^d}(|\xi|^2+\lambda_{ p} )|\widehat{\psi}(\xi, p)|^2+ \tfrac{1}{4} \int_{\xi \in \mathbb{R}} \sum_{|\omega| < \frac{1}{M}}\sum_{(p, q, r, s) \in \Gamma_{\omega}} \widehat{\psi}(\xi, q) \widehat{\psi}(\xi, s) \overline{\widehat{\psi}}(\xi, r) \overline{\widehat{\psi}}(\xi, p) .
	\end{align}
Since $\theta \in \Theta_1$, there exists $\gamma > d$ such that the following Diophantine condition on $\theta$ holds:
    \begin{align*}
        \left| \sum_{i=1}^d \theta_i^2 n_i \right| \geq C_{\theta}\|(n_1, ..., n_d)\|^{-\gamma}
    \end{align*}
    implies that there exists a constant $C_d$ such that if $\theta \in \Theta_1$ and $\tfrac{1}{M}>|\lambda_p - \lambda_q + \lambda_r - \lambda_s| \neq 0$, then 
    \begin{align}
        \max ( \|p\|, \|r\|, \|q\|, \|s\|) \geq C_{d, \theta}M^{\frac{1}{2\gamma}}.
    \end{align}
 Moreover, if  $\max ( \|p\|, \|r\|, \|q\|, \|s\|) \leq C_{d, \theta}M^{\frac{1}{2\gamma}}$ and $\tfrac{1}{M}>|\lambda_p - \lambda_q + \lambda_r - \lambda_s|$, then  
    \begin{align}
        \lambda_p - \lambda_q + \lambda_r - \lambda_s=0.
    \end{align}
    Let $v_{p, \xi}:= \widehat{\psi}(\xi, p)$. The  system of equations associated to $H_M$ is then
	\begin{align}\label{gaugeNLS}
		i\tfrac{d}{dt}v_{ \xi, p} =(|\xi|^2+\lambda_{ p})v_{\xi, p} + \sum_{(p, q, r, s) \in \Gamma_0^K} v_{q, \xi}v_{\xi, s}\overline{v}_{\xi, r}+  \sum_{|\omega|< \frac{1}{M}}\sum_{ (p, q, r, s) \in \Gamma_{\omega} \atop \max(\|p\|, \|q\|, \|r\|, \|s\|)>K} v_{\xi, q}v_{\xi, s}\overline{v}_{\xi, r}.
	\end{align}
	where 
	\begin{align} \label{eq K}
	    K = C_{d, \theta}M^{\frac{1}{2\gamma}}.
	\end{align}
A gauge transformation allows us to reduce this system to 
	\begin{align}\label{truncnls}
		i\frac{d}{dt}u_{\xi, p}= \sum_{(p, q, r, s) \in \Gamma_0^K} u_{\xi, q}u_{\xi, s}\overline{u}_{\xi, r}+  \sum_{|\omega|< \frac{1}{M}}\sum_{ (p, q, r, s) \in \Gamma_{\omega} \atop \max(\|p\|, \|q\|, \|r\|, \|s\|)>K} e^{it \omega} u_{\xi, q}u_{\xi, s}\overline{u}_{\xi, r} .
	\end{align}
	
We will first show that a solution to \eqref{truncnls} is  analytic as a function from $[0, T]$ to $H^k$ for $k>d/2$, where $T$ is only required to be positive.

Analyticity will essentially follow from the fact that we are assuming enough regularity to imply that $h^k$ is an algebra.  The full proof (for the $d=2$ case) appears in \cite{SW20} and \cite{CF12}.

\begin{lemma}\label{lemmaT}
Let $k> d/2$ and $u(0) = \{u_{\xi, p}(0)\}_{(\xi, p) \in \mathbb{R} \times\mathbb{Z}^d} \in H^k( \mathbb{R} \times\mathbb{Z}^d)$. There exists $T>0$ and a unique analytic in time solution $\{u_{\xi, p}\} : [0,T] \rightarrow H^k$ to \eqref{truncnls}. Moreover, there exist constants $C$ and $R$ such that for all $n \in \mathbb{N}$ and $\tau \leq T$,
	\begin{align} \label{expbound}
		\left|   \frac{d^nu_{\xi, p}}{(dt)^n}(\tau) \right| \leq CR^nn!.
	\end{align}
Moreover, $T$ can be taken to be $T = C_{d,k}|u(0)|^{-2}_{k}$, for some absolute constant $C_{d, k}>0$.
\end{lemma}

Define the $K$-box in $\mathbb{Z}^d$ by
	\begin{align*}
		B_K = \left\{ p \in \mathbb{Z}^d~:~ \max(|p_1|,...,|p_d|)\leq K \right\}.
	\end{align*}
We state the following lemma that appears in \cite{SW20}. This lemma provides a description of the stability of the Fourier support of frequency-localized solutions to the resonant equation associated to the Hamiltonian, $H_M$, defined in \eqref{eq ResHam}.
\begin{lemma}\label{nocasc}
	Let $K >0$ and assume $ u_{\xi, p}(0)=0$ for $ p \not\in B_{K/3}$ and let $\{u_{\xi, p}\}: [0,T] \rightarrow H^k$ be a solution to \eqref{truncnls}.  If $p \in  \mathbb{Z}^d \setminus B_{K/3}$ then for any $n \in \mathbb{Z}_+$,
	\begin{align*}
		  \frac{d^nu_{\xi, p}}{(dt)^n}(0)=0.
	\end{align*}
\end{lemma}
The proof of this lemma (in the two-dimensional case) appears in \cite{SW20}.

Lemma \ref{nocasc} leads directly to the following corollary by Taylor expansion.
\begin{corollary}\label{nogrowth}
		Let $K >0$ and assume $ u_{\xi, p}(0)=0$ for $p  \not\in B_{K/3}$ and let $\{u_{\xi, p}\}: [0,T] \to H^k$ be an analytic solution to \eqref{truncnls}.  If $p \in  \mathbb{Z}^d \setminus B_{K/3}$, then 
	\begin{align*}
		  u_{\xi, p}(t)=0
	\end{align*}
		for $t \in [0, T]$.
	\end{corollary}

Lemma \ref{nocasc} and Corollary \ref{nogrowth} are proven in \cite{SW20}, so the proofs of these results are not included here. The next step of the process is the following symmetry result:

\begin{proposition}\label{nocascprop}
	Let $K >0$ and assume $v_{\xi, p}^0=0$ for  $p \not\in B_{K/3}$. Let $\{v_{\xi, p}\}: [0,T] \to H^k$ be an analytic solution to  \eqref{gaugeNLS} with $v_{\xi, p}(0)=v^0_{\xi, p}$ such that for all $t \in [0, T]$
		\begin{align*}
			\frac{d}{dt} |v|^2_{k} =0,
		\end{align*}	
	where the $k$-norm is defined in Definition \ref{defn k}, and it is  equivalent to the Sobolev $H^k$-norm.

\end{proposition}

The proof is very similar to that which appears in \cite{HPSW21}, so the details are relegated to Appendix \ref{app.nocasc}.

\begin{theorem}\label{Thm.constant}
    Let $K >0$,  assume $v_{\xi, p}^0=0$ for  $p \not\in B_{K/3}$, and let $T>0$. There exists
		$\{v_{\xi, p}\}: [0,T] \rightarrow H^k$ which is an unique solution to \eqref{gaugeNLS} with $v_{\xi, p}(0)=v^0_{\xi, p}$ such that for all $t \in [0, T]$
		\begin{align*}
			 |v(t)|^2_{k} =|v(0)|^2_{k},
		\end{align*}	
for all $t \in [0, T]$.

\end{theorem}

\begin{remark}[Proof of Theorem \ref{Thm.constant}]
    Theorem \ref{Thm.constant} follow from first observing that, by Lemma \ref{lemmaT}, the time of analyticity, depends on the $k$-norm of the initial data. Then Proposition \ref{nocascprop} implies that on the analytic times of existence the $k$-norm is constant. Therefore at the end of each interval of analyticity, one can start a new interval of the same length. Therefore, over any fixed time interval, $[0, T]$, the constancy of the $k$-norm given by the conclusion of Proposition \ref{nocascprop} holds.
\end{remark}

\section{Asymptotic Behavior}\label{section.asymptotic}

In this section, we will use the theory of modified scattering to relate the behavior, demonstrated in Section \ref{sec.dynamics}, of solutions to  the effective systems to solutions of the full cubic NLS system \eqref{eq NLS} on $\mathbb{R} \times \mathbb{T}_{\theta}^d$. We begin with the main lemma which provides asymptotic stability between solutions of NLS and solutions to the effective equation. 

\begin{lemma}\label{asymptotic}
	Let $M\geq 0$, $0<\varepsilon_2< \varepsilon_1< 1$ and let $U(t)$ be a solution to \eqref{eq NLS}  on the time interval $[2^n, 2^{n+1})$ and let $G_n: [2^n, 2^{n+1}) \to H^N$ be a solution to the equation
	\begin{align}\label{eq effect}
		i\partial_t G_n = \frac{\pi}{t} \mathcal{R}^t_{M}[G_n, G_n, G_n] . 
	\end{align}

If $\|U\|_{X_T^+} \leq \varepsilon_1$ for $T\geq 2^{n+1}$ and $\|e^{-i2^n \Delta}U(2^n)- G_n(2^n)\|_{H^N} \leq \varepsilon_2$, then  for $t \in [2^n, 2^{n+1})$	
\begin{align*}
		\|e^{-it \Delta}U(t)- G_n(t)\|_{H^N} \lesssim \varepsilon_2  + \varepsilon_1^3 2^{-2 n \delta}  + \varepsilon_1^2(\varepsilon_2 + \varepsilon_1^3 2^{-2 n \delta}) 2^{n\delta}.
	\end{align*}
The above inequality holds uniformly in $M\geq 0$.
\end{lemma}

The proof of Lemma \ref{asymptotic} follows the same strategy as the equivalent stability argument in the proof of Theorem 6.1 in \cite{HPTV15}.
\begin{proof}[Proof of Lemma \ref{asymptotic}]
Since the $H^k$ norms of $G_n$ are bounded and $\varepsilon_2<\varepsilon_1$, triangle inequality implies that $\|G_n\|_{X^+_T}\lesssim \varepsilon_1$.
This proof consists of two Duhamel formula steps. First, observe that 
    \begin{align}
        i \partial_t (e^{-i t\Delta } U) = \mathcal{N}^t[e^{-i t\Delta } U,e^{-i t \Delta } U,e^{-i t \Delta } U] .
    \end{align}
    
    Then we have the following Duhamel Formula
\begin{align}
e^{-i t\Delta } U(t) & = e^{-it \Delta } U(2^n) - i \int_{2^n}^t \mathcal{N}^{\tau} [e^{-i \tau\Delta } U, e^{-i \tau\Delta } U, e^{-i \tau\Delta } U] \, d\tau ,\\
G_n (t) & = G_n(2^n) - i\int_{2^n}^t \frac{\pi}{\tau} \mathcal{R}_M^{\tau} [G_n ,G_n, G_n] \, d\tau .
\end{align}

From here we aim to establish an application of Gronwall's Lemma.
\begin{align}
\norm{e^{-i t \Delta } U (t) - G_n (t)}_{Z_d} & \leq \norm{e^{-i t\Delta } U(2^n) - G_n(2^n)}_{Z_d} + \norm{\int_{2^n}^t \mathcal{N}^{\tau} [e^{-i \tau\Delta } U] - \frac{\pi}{\tau} \mathcal{R}_M^{\tau} [G_n] \, d \tau}_{Z_d} \\
& \leq \varepsilon_2 + \norm{\int_{2^n}^t \mathcal{N}^{\tau} [e^{-i \tau\Delta } U] -\frac{\pi}{\tau} \mathcal{R}_M^{\tau}[e^{-i \tau\Delta } U] \, d \tau}_{Z_d}  + \norm{\int_{2^n}^t \frac{\pi}{\tau} \mathcal{R}_M^{\tau} [e^{-i \tau\Delta } U] - \frac{\pi}{\tau} \mathcal{R}_M^{\tau} [G_n ] \, d\tau}_{Z_d} .
\end{align}
where we used the convention $\mathcal{N}^{\tau} [F] =\mathcal{N}^{\tau} [F,F,F]$ and $\mathcal{R}_M^{\tau} [F] = \mathcal{R}_M^{\tau} [F,F,F]$.

We can bound the second term using the last estimate in the statement of Proposition \ref{prop 3.1}. The third term is bounded using a simple trilinear argument along with the smallness condition assumed for the $X_T^+$-norm of $U$ and a smallness condition  on the norm of $G_n$. The smallness of the norm of $G_n$ follows from the fact that $\|G_n(t)\|_{h^k} \leq C_d \|G_n(2^n)\|_{h^k} \lesssim \varepsilon$ (for $t\geq 2^n$) which follows from Theorem \ref{Thm.constant} and the triangle inequality. Therefore,
\begin{align}
\norm{e^{-i t \Delta } U (t) - G_n (t)}_{Z_d}& \lesssim \varepsilon_2 + \varepsilon_1^3 2^{-2 n \delta} + \int_{2^n}^t \varepsilon_1^2 \norm{e^{-i\tau \Delta} U - G_n}_{Z_d} \, \frac{d\tau }{\tau} .
\end{align}
    Gronwall's lemma then implies that 

\begin{align}
   \norm{e^{-i t \Delta } U (t) - G_n (t)}_{Z_d}& \lesssim (\varepsilon_2 + \varepsilon_1^3 2^{-2 n \delta}) e^{\varepsilon_1^2 \log(t/2^n)}\lesssim \varepsilon_2 + \varepsilon_1^3 2^{-2 n \delta}
\end{align}
since $t/2^n \leq 2.$

Similarly,
\begin{align}
\norm{e^{-i t \Delta } U (t) - G_n (t)}_{S} & \leq \norm{e^{-i t\Delta } U(2^n) - G_n(2^n)}_{S} + \norm{\int_{2^n}^t \mathcal{N}^{\tau} [e^{-i \tau\Delta } U] - \frac{\pi}{\tau} \mathcal{R}_M^{\tau} [G_n] \, d \tau}_{S} \\
& \leq \varepsilon_2 + \norm{\int_{2^n}^t \mathcal{N}^{\tau} [e^{-i \tau\Delta } U] -\frac{\pi}{\tau} \mathcal{R}_M^{\tau}[e^{-i \tau\Delta } U] \, d \tau}_{S} + \norm{\int_{2^n}^t \frac{\pi}{\tau} \mathcal{R}_M^{\tau} [e^{-i \tau\Delta } U] - \frac{\pi}{\tau} \mathcal{R}_M^{\tau} [G_n ] \, d\tau}_{S} ,
\end{align}
and
    \begin{align}
&\norm{e^{-i t \Delta } U (t) - G_n (t)}_{S}\\
& \hspace{.7cm}\lesssim \varepsilon_2 + \varepsilon_1^3 2^{-2 n \delta} + \int_{2^n}^t \varepsilon_1(\|U\|_S+ \|G_n\|_S) \norm{e^{-i\tau \Delta} U - G_n}_{Z_d} \, \frac{d\tau }{\tau} + \int_{2^n}^t \varepsilon_1^2 \norm{e^{-i\tau \Delta} U - G_n}_{S} \, \frac{d\tau }{\tau} \\
&\hspace{.7cm}\lesssim \varepsilon_2  + \varepsilon_1^3 2^{-2 n \delta}  + \varepsilon_1^2(\varepsilon_2 + \varepsilon_1^3 2^{-2 n \delta}) 2^{n\delta} +\int_{2^n}^t \varepsilon_1^2 \norm{e^{-i\tau \Delta} U - G_n}_{S} \, \frac{d\tau }{\tau}.
\end{align}

In the last step of the previous inequality, we are using that $(1+|t|)^{-\delta}(\|U\|_S+\|G\|_S) \leq \varepsilon$ by the definition of the $X_T^+$-norm. Again Gronwall implies 
\begin{align}
   \norm{e^{-i t \Delta } U (t) - G_n (t)}_{S}& \lesssim (\varepsilon_2  + \varepsilon_1^3 2^{-2 n \delta}  + \varepsilon_1^2(\varepsilon_2 + \varepsilon_1^3 2^{-2 n \delta}) 2^{n\delta}) 2^{\varepsilon_2^2 \log(t/2^n)}\\
   &\lesssim \varepsilon_2  + \varepsilon_1^3 2^{-2 n \delta}  + \varepsilon_1^2(\varepsilon_2 + \varepsilon_1^3 2^{-2 n \delta}) 2^{n\delta} .
\end{align}
\end{proof}

\begin{proposition}\label{prop 6.2}
There exists $\varepsilon > 0$ such that any data $U_0 \in S^+$ satisfying 
\begin{align}
\norm{U_0}_{S^+} \leq \varepsilon
\end{align}
generates a global solution $U(t)$ of \eqref{eq NLS}.  Moreover, for any $T > 0$, it holds that 
\begin{align}
\norm{F}_{X_T^+} \leq 2 \varepsilon ,
\end{align}
where $F(t) = e^{-it \Delta_{\R \times \T^d_{\theta}}} U(t)$. 
\end{proposition}

\begin{proof}[Proof of Proposition \ref{prop 6.2}] 

Let $\mathcal{F} N^{\tau} [\widehat{F}_p (\xi ,\tau) ]=\mathcal{F} N^{\tau} [\widehat{F}_p (\xi ,\tau) , \widehat{F}_p (\xi ,\tau) , \widehat{F}_p (\xi ,\tau)]$. Then a basic computation yields
\begin{align}\label{eq 16}
\frac{1}{2} \frac{d}{d\tau}  \norm{\widehat{F}_p (\xi , \tau)}_{h_p^{\frac{d}{2}+}}^2  & = \inner{\widehat{F}_p (\xi ,\tau) , \frac{d}{d\tau} \widehat{F}_p (\xi , \tau)}_{h_p^{\frac{d}{2}+} \times h_p^{\frac{d}{2}+}}  = \inner{\widehat{F}_p (\xi , \tau) , \mathcal{F} N^{\tau} [\widehat{F}_p (\xi ,\tau) ]}_{h_p^{\frac{d}{2}+} \times h_p^{\frac{d}{2}+}} .
\end{align}
From here on, in this computation, $M(\tau)$ is a piecewise constant function of $\tau$. In particular,
\begin{align}
M(\tau) = 2^n, \quad n = \lfloor \ln (\tau^{\delta})  \rfloor ,
\end{align} 
where $\lfloor  x  \rfloor$ is the integer part of $x$. Furthermore, let $K (\tau)= C_{d, \theta}M(\tau)^{\frac{1}{2\gamma}}$ be defined as in Equation \eqref{eq K}.  For the moment, we will suppress the dependence of $M$ on $\tau.$
\begin{align}
\eqref{eq 16} & =\inner{\widehat{F}_p (\xi , \tau) , \mathcal{F} N^{\tau} [\widehat{F}_p (\xi ,\tau) ]}_{h_p^{\frac{d}{2}+} \times h_p^{\frac{d}{2}+}}  = \inner{ L_{> K/3}  \widehat{F}_p (\xi , \tau) + L_{\leq K/3} \widehat{F}_p (\xi , \tau)  , \mathcal{F} N^{\tau} [\widehat{F}_p (\xi ,\tau) ] ) }_{h_p^{\frac{d}{2}+} \times h_p^{\frac{d}{2}+}} \\
& =  \inner{L_{> K/3}  \widehat{F}_p (\xi , \tau) , \mathcal{F} N^{\tau} [\widehat{F}_p (\xi ,\tau) ] }_{h_p^{\frac{d}{2}+} \times h_p^{\frac{d}{2}+}} + \inner{L_{\leq K/3} \widehat{F}_p (\xi , \tau) ,\mathcal{F} N^{\tau} [\widehat{F}_p (\xi ,\tau) ]}_{h_p^{\frac{d}{2}+} \times h_p^{\frac{d}{2}+}} \\
& = : A + B \label{eq 17},
\end{align}
where we recall that  $L_{>K/3}$ and $L_{\leq K/3}$ are Littlewood-Paley projections on the frequency in $\T_{\theta}^d$ only.

We bound the first term, $A$, in \eqref{eq 17} via Cauchy-Schwarz and  Bernstein's inequality
    \begin{align}
       & \inner{L_{> K/3}  \widehat{F}_p (\xi , \tau) , \mathcal{F} N^{\tau} [\widehat{F}_p (\xi ,\tau) ] }_{h_p^{\frac{d}{2}+} \times h_p^{\frac{d}{2}+}} \leq \norm{L_{> K/3}  \widehat{F}_p}_{h_p^{\frac{d}{2}+}} \norm{\mathcal{N}^{\tau}}_{h_p^{\frac{d}{2}+}} \\
& \hspace{.7cm}\lesssim     M(\tau)^{-\frac{1}{2\gamma}(N - \frac{d}{2}-)} \norm{\widehat{F}_p}_{h_p^N}  (1 + \abs{\tau})^{-1 } \norm{\widehat{F}_p}_{h_p^{\frac{d}{2}+}}^3 \lesssim      \norm{\widehat{F}_p}_{S}  (1 + \abs{\tau})^{-1 -\frac{1}{2\gamma}\delta(N-\frac{d}{2}-)} \norm{\widehat{F}_p}_{h_p^{\frac{d}{2}+}}^3 \label{eq 3} .
\end{align}
By Definition \ref{defn XT}, $\norm{\widehat{F}_p(\xi, \tau)}_{S} \leq (1+|\tau|)^{\delta}\norm{\widehat{F}_p(\xi, \tau)}_{X_T^+}$ and  $\norm{\widehat{F}_p}_{h^{\frac{d}{2}+}} \leq\norm{\widehat{F}_p}_{X_T^+}$. 

Therefore, if $N-\frac{d}{2}- \geq 4\gamma$, then 
 \begin{align}
      \inner{L_{> K/3}  \widehat{F}_p (\xi , \tau) , \mathcal{F} N^{\tau} [\widehat{F}_p (\xi ,\tau) ] }_{h_p^{\frac{d}{2}+} \times h_p^{\frac{d}{2}+}}  
        &\lesssim  \norm{\widehat{F}_p}_{h_p^{\frac{d}{2}+}} \norm{\widehat{F}_p}_{X_T^+}^3(1 + \abs{\tau})^{-1 -\delta} .
  \end{align}

For the B term in \eqref{eq 17}, we recall the decomposition in \eqref{eq Decomp}
    \begin{align}
        \mathcal{F} N^{\tau}=\Pi^{\tau}_M+ \mathcal{E}_{1, M}^{\tau}   + \mathcal{E}_{2, M}^{\tau}=  (\Pi^{\tau}_M - \frac{\pi}{\tau}\mathcal{R}_M^{\tau})+ \mathcal{E}_{1, M}^{\tau}   + \mathcal{E}_{err, M}^{\tau}+ \partial_{\tau} \mathcal{E}_{3,M}^{\tau}+\frac{\pi}{\tau} \mathcal{R}_M^{\tau}.
    \end{align}
  Therefore,
\begin{align}
B =&\inner{L_{\leq K/3} \widehat{F}_p  ,\mathcal{F} N^{\tau} [\widehat{F}_p (\xi ,\tau) ]}_{h_p^{\frac{d}{2}+} \times h_p^{\frac{d}{2}+}}\\
&\hspace{.5cm}=\inner{L_{\leq K/3} \widehat{F}_p  , (\Pi^{\tau}_M - \frac{\pi}{\tau}\mathcal{R}_M^{\tau})  }_{h_p^{\frac{d}{2}+} \times h_p^{\frac{d}{2}+}}+\inner{L_{\leq K/3} \widehat{F}_p  , (  \mathcal{E}_{1, M}^{\tau}   + \mathcal{E}_{err, M}^{\tau} )  }_{h_p^{\frac{d}{2}+} \times h_p^{\frac{d}{2}+}}\\
&\hspace{1cm}+\inner{L_{\leq K/3} \widehat{F}_p  , \partial_{\tau} \mathcal{E}_{3,M}^{\tau}  }_{h_p^{\frac{d}{2}+} \times h_p^{\frac{d}{2}+}}+\inner{L_{\leq K/3} \widehat{F}_p  , \frac{\pi}{\tau} \mathcal{R}_M^{\tau}[\widehat{F}_p (\xi ,\tau) ] }_{h_p^{\frac{d}{2}+} \times h_p^{\frac{d}{2}+}} \\
& \hspace{.5cm} =: I+II+III+IV \label{eq 18}.
\end{align}
 Consider expression I in \eqref{eq 18}, by Lemma \ref{lem 3.7} and Definition \ref{defn XT}
 \begin{align}
       I = \inner{L_{\leq K/3} \widehat{F}_p  , (\Pi^{\tau}_M - \frac{\pi}{\tau}\mathcal{R}_M^{\tau})  }_{h_p^{\frac{d}{2}+} \times h_p^{\frac{d}{2}+}} & \leq\norm{ L_{\leq K/3} \widehat{F}_p }_{h_p^{\frac{d}{2}+}} \norm{\Pi^{\tau}_M - \frac{\pi}{\tau}\mathcal{R}_M^{\tau}}_{h_p^{\frac{d}{2}+}} \\
      &\lesssim   \norm{\widehat{F}_p}_{h_p^{\frac{d}{2}+}} (1 + \abs{\tau})^{-1 - 17\delta} \norm{\widehat{F}_p}_{X_T^+}^3  .
  \end{align}
 We now use Proposition \ref{prop 3.1} for expression II from \eqref{eq 18}, 
  \begin{align}
    II = \inner{L_{\leq K/3} \widehat{F}_p  , (  \mathcal{E}_{1, M}^{\tau}   + \mathcal{E}_{err, M}^{\tau} )  }_{h_p^{\frac{d}{2}+} \times h_p^{\frac{d}{2}+}}
     &\leq \norm{ L_{\leq K/3} \widehat{F}_p }_{h_p^{\frac{d}{2}+}} \norm{\mathcal{E}_{1,M}^{\tau}+\mathcal{E}_{err, M}^{\tau}}_{h_p^{\frac{d}{2}+}} \\
     &\lesssim   \norm{\widehat{F}_p}_{h_p^{\frac{d}{2}+}} (1 + \abs{\tau})^{-1 - \delta} \norm{\widehat{F}_p}_{X_T^+}^3 .
  \end{align}
Again, by Proposition \ref{prop 3.1} and Definition \ref{defn XT}, we estimate expression III in  \eqref{eq 18} by estimating three terms associated to the integration-by-parts calculation,
  \begin{align}
       \inner{L_{\leq K/3} \partial_{\tau} \widehat{F}_p ,  \mathcal{E}_{3,M}^{\tau}}_{h_p^{\frac{d}{2}+} \times h_p^{\frac{d}{2}+}} &\lesssim    \norm{\partial_{\tau} \widehat{F}_p}_{h_p^{\frac{d}{2}+}} M (\tau) (1 + \abs{\tau})^{-\frac{1}{10}} \norm{\widehat{F}_p}_{X^+_T}^2 \norm{ \widehat{F}_p}_{h_p^{\frac{d}{2}+}}\\
       &\lesssim   \sup_{\tau \in [0, t]} \norm{\widehat{F}_p}_{h_p^{\frac{d}{2}+}}\norm{\widehat{F}_p}_{X^+_T}^3  M (\tau) (1 + \abs{\tau})^{-1-\frac{1}{10}+7\delta}   .
  \end{align}
  \begin{align}
       \inner{L_{\leq K/3}  \widehat{F}_p ,  \mathcal{E}_{3,M}^{t}}_{h_p^{\frac{d}{2}+} \times h_p^{\frac{d}{2}+}} &\lesssim    \norm{ \widehat{F}_p}_{h_p^{\frac{d}{2}+}} M (t) (1 + \abs{t})^{-\frac{1}{10}} \norm{\widehat{F}_p}_{X^+_T}^2 \norm{ \widehat{F}_p}_{h_p^{\frac{d}{2}+}}.
  \end{align}
  \begin{align}
       \inner{L_{\leq K/3}  \widehat{F}_p ,  \mathcal{E}_{3,M}^{0}}_{h_p^{\frac{d}{2}+} \times h_p^{\frac{d}{2}+}} &\lesssim    M (0)\norm{ \widehat{F}_p(\xi, 0)}_{h_p^{\frac{d}{2}+}}^4 .
  \end{align}
  
Expression IV in \eqref{eq 18} requires a different argument that is not unlike that used for the second term, but with an additional identity from Appendix \ref{app.nocasc}. We first decompose in the following way
    \begin{align}
      IV=   \inner{L_{\leq K/3} \widehat{F}_p  , \frac{\pi}{\tau} \mathcal{R}_M^{\tau}[\widehat{F}_p (\xi ,\tau) ] }_{h_p^{\frac{d}{2}+} \times h_p^{\frac{d}{2}+}}&=\frac{\pi}{\tau} \inner{L_{\leq K/3} \widehat{F}_p  , \sum_{|\omega|<\frac{1}{M}} \sum_{(p, q, r, s) \in \Gamma_{\omega}} e^{i \tau \omega}\widehat{F}_q\overline{\widehat{F}}_r\widehat{F}_s }_{h_p^{\frac{d}{2}+} \times h_p^{\frac{d}{2}+}}\\
        &=\frac{\pi}{\tau} \inner{L_{\leq K/3} \widehat{F}_p  ,  \sum_{(p, q, r, s) \in \Gamma_{0}^K} e^{i \tau \omega}\widehat{F}_q\overline{\widehat{F}}_r\widehat{F}_s }_{h_p^{\frac{d}{2}+} \times h_p^{\frac{d}{2}+}}\\
        &+\frac{\pi}{\tau} \inner{L_{\leq K/3} \widehat{F}_p  , \sum_{|\omega|<\frac{1}{M}} \sum_{(p, q, r, s) \in \Gamma_{\omega} \atop \max(\|p\|, \|q\|, \|r\|, \|s\|)>K} e^{i \tau \omega}\widehat{F}_q\overline{\widehat{F}}_r\widehat{F}_s}_{h_p^{\frac{d}{2}+} \times h_p^{\frac{d}{2}+}} .
    \end{align}
We will use Bernstein's inequality and basic trilinear estimates to bound the second term in IV. The first term in IV needs to be decomposed further:
\begin{align*}
   \frac{\pi}{\tau} \inner{L_{\leq K/3} \widehat{F}_p  ,  \sum_{(p, q, r, s) \in \Gamma_{0}^K} e^{i \tau \omega}\widehat{F}_q\overline{\widehat{F}}_r\widehat{F}_s }_{h_p^{\frac{d}{2}+} \times h_p^{\frac{d}{2}+}}
    &=\frac{\pi}{\tau} \inner{L_{\leq K/3} \widehat{F}_p  ,  \sum_{(p, q, r, s) \in \Gamma_{0}^{K/3}} e^{i \tau \omega}\widehat{F}_q\overline{\widehat{F}}_r\widehat{F}_s }_{h_p^{\frac{d}{2}+} \times h_p^{\frac{d}{2}+}}
    \\&+ \frac{\pi}{\tau} \inner{L_{\leq K/3} \widehat{F}_p  ,  \sum_{(p, q, r, s) \in \Gamma_{0}^K\atop \max(\|p\|, \|q\|, \|r\|, \|s\|)>K/3} e^{i \tau\omega}\widehat{F}_q\overline{\widehat{F}}_r\widehat{F}_s }_{h_p^{\frac{d}{2}+} \times h_p^{\frac{d}{2}+}} .
\end{align*}
By Claim \ref{claim constant}, 
    \begin{align}
        \frac{\pi}{\tau} \inner{L_{\leq K/3} \widehat{F}_p  ,  \sum_{(p, q, r, s) \in \Gamma_{0}^{K/3}} e^{i \tau \omega}\widehat{F}_q\overline{\widehat{F}}_r\widehat{F}_s }_{h_p^{\frac{d}{2}+} \times h_p^{\frac{d}{2}+}}=0 .
    \end{align}
Furthermore, as described above
    \begin{align}
        \left|\frac{\pi}{\tau} \inner{L_{\leq K/3} \widehat{F}_p  ,  \sum_{(p, q, r, s) \in \Gamma_{0}^K\atop \max(\|p\|, \|q\|, \|r\|, \|s\|)>K/3} e^{i \tau \omega}\widehat{F}_q\overline{\widehat{F}}_r\widehat{F}_s }_{h_p^{\frac{d}{2}+} \times h_p^{\frac{d}{2}+}}\right| \lesssim \frac{1}{\tau}\|L_{> K/3} \widehat{F}\|_ {h_p^{\frac{d}{2}+}}\|\widehat{F}\|_{h_p^{\frac{d}{2}+}}^3
    \end{align}
    and
    \begin{align}
        \left| \frac{\pi}{\tau} \inner{L_{\leq K/3} \widehat{F}_p  ,  \sum_{|\omega|<\frac{1}{M}} \sum_{(p, q, r, s) \in \Gamma_{\omega} \atop \max(\|p\|, \|q\|, \|r\|, \|s\|)>K} e^{i \tau\omega}\widehat{F}_q\overline{\widehat{F}}_r\widehat{F}_s }_{h_p^{\frac{d}{2}+} \times h_p^{\frac{d}{2}+}} \right| \lesssim \frac{1}{\tau}\|L_{> K} \widehat{F}\|_ {h_p^{\frac{d}{2}+}}\|\widehat{F}\|_ {h_p^{\frac{d}{2}+}}^3 .
    \end{align}
    We finally have that
    \begin{align}
      IV=  \inner{L_{\leq K/3} \widehat{F}_p  , \frac{\pi}{\tau} \mathcal{R}_M^{\tau}[\widehat{F}_p (\xi ,\tau) ] }_{h_p^{\frac{d}{2}+} \times h_p^{\frac{d}{2}+}}  &\lesssim  M(\tau)^{-\frac{1}{2\gamma}(N - \frac{d}{2}-)} \norm{\widehat{F}_p}_{h_p^N}  (1 + \abs{\tau})^{-1 } \norm{\widehat{F}_p}_{h_p^{\frac{d}{2}+}}^3 \\
        &\lesssim  (1 + \abs{\tau})^{-1-\delta}  \norm{\widehat{F}_p}_{h_p^{\frac{d}{2}+}}   \norm{\widehat{F}_p}_{X_T^+}^3 .
    \end{align}
This completes the estimate for term B from equation \eqref{eq 17}. Combining the previous estimates, we have shown that
    \begin{align}
       \eqref{eq 16} = \tfrac{1}{2} \tfrac{d}{d\tau}  \norm{\widehat{F}_p (\xi , \tau)}_{h_p^{\frac{d}{2}+}}^2\lesssim  (1 + \abs{\tau})^{-1-\delta}  \norm{\widehat{F}_p(\xi, \tau)}_{h_p^{\frac{d}{2}+}}   \norm{\widehat{F}_p}_{X_T^+}^3 .
    \end{align}
Using the fundamental theorem of calculus, we have
\begin{align}
 \abs{\norm{\widehat{F}_p (\xi , t)}_{Z_d}-  \norm{\widehat{F}_p (\xi , 0)}_{Z_d}} &= \abs{\int_0^t \tfrac{d}{d\tau}  \norm{\widehat{F}_p (\xi , \tau)}_{Z_d} \, d\tau} \lesssim  \norm{\widehat{F}_p}_{X_T^+}^3.
\end{align}

Finally, Proposition \ref{prop 3.1} and Remark \ref{rmk Replacement} provide the necessary bound for the $S$-norm:
\begin{align}
    \norm{F(t)-F(1)}_{S} \lesssim \norm{\int_{1}^t  \mathcal{N}^{\tau}-\frac{\pi}{\tau}\mathcal{R}_M^{\tau} \, d\tau}_{S} + \norm{\int_{1}^t  \frac{\pi}{\tau}\mathcal{R}_M^{\tau} \, d\tau}_{S} \lesssim |t|^{\delta}\norm{\widehat{F}_p}_{X_T^+}^3 .
\end{align}

These estimates suffice to construct a contraction mapping existence and uniqueness argument as long as the $X_T^+$ norm and the time-zero $S^+$ norm are small enough.

\end{proof}

For the rest of this section, we are assuming that $\theta \in \Theta_1$ (recall Definition \ref{def.Theta}). Recall from Subsection \ref{subsec.boundeddynamics}, since $\theta \in \Theta_1$, there exists $\gamma>d$ such that the following Diophantine condition holds:
        \begin{align*}
        \left| \sum_{i=1}^d \theta_i^2 n_i \right| \geq C_{\theta}\|(n_1, ..., n_d)\|^{-\gamma}.
    \end{align*}

		\begin{theorem}\label{thm.asym}
			Let $d \geq 2$ and let $N>10 (d +\gamma)$.  There exists $\varepsilon(N, d)>0$ such that  if $\norm{U_0}_{S^+}<\varepsilon(N, d)$ and $U(t)$ solves \eqref{eq NLS} with initial data $U_0$, then there exists a sequence $\{G_j\}^{\infty}_{j=1}$ such that 
			\begin{align*}
				G_j: [2^{j}, 2^{j+1}) \to H^N
			\end{align*}
such that if $G(t) = \sum_{j=1}^{\infty} G_j(t)$, then 
			\begin{align*}
				\|e^{-it \Delta_{\R \times \T_{\theta}^d}} U(t) - G(t) \|_{H^{N-1}} \to 0
			\end{align*}
as $t \to \infty$. Moreover, there exists $T_U$, such that 
			\begin{align*}
			\sup_{t \in [0, \infty)}\|U\|_{H^{N-1}} \leq C_1\varepsilon^5 +
			C_2\sup_{[0, T_U]}\|U(t)\|_{H^N}.
			\end{align*}
		\end{theorem}

\begin{proof}[Proof of Theorem \ref{thm.asym}]

	Let $\{M_n\}_{n=1}^{\infty} \subset \mathbb{R}_+$ and $\{K_n\}_{n=1}^{\infty} \subset \mathbb{R}_+$  defined by 
	    \begin{align}\label{eq 21}
	    \begin{aligned}
            M_n&:=2^{20\gamma\delta n} ,\\
	        K_n&:=C_{ d, \theta}2^{10 \delta n}= C_{ d, \theta}M_N^{\frac{1}{2\gamma}}.
	      \end{aligned}
	    \end{align} 
	   The relationship between $K_n$ and $M_n$ is the same as given by equation \eqref{eq K} and will serve the same purpose. 
	   
	  Let $\norm{U_0}_{S^+}<\varepsilon$ where epsilon is given by Proposition \ref{prop 6.2}. Let $F(t) = e^{-it\Delta_{\R \times \T_{\theta}^d}} U(t)$ where $U(t)$ is a global solution to \eqref{eq NLS} with $U(0)=U_0$. For each $n \in \mathbb{Z}_+$, define $g_n(n)$ as a Fourier cutoff (in the discrete Fourier variables) of $F(2^n)$ at level $K_n/3$:
	\begin{align*}
		\widehat{g}_n(\pi (\ln 2) n , \xi, p) := \left\{ \begin{array}{ll}  0 & |p|>K_n/3 \\ \widehat{F}(2^n, \xi, p) & |p|\leq K_n/3 .\end{array}  \right.
	\end{align*}

It follows immediately that 
\begin{align}
 \|g_n(\pi (\ln 2)n) - F(2^n)\|_{H^{N-1}} \lesssim K_n^{-1}\|F(2^n)\|_{H^N}    .
\end{align}
Now let $g_n: [\pi(\ln 2) n, \pi(\ln 2)(n+1)) \to H^N$ be a solution to the equation associated to the Hamiltonian, $H_{M_n}$, defined in \eqref{eq ResHam}
	\begin{align*}
		H_{M_n}(\psi)&= H_0(\psi)+N_{M_n}(\psi)\\
				&=\tfrac{1}{2} \int_{\xi \in \mathbb{R}}  \sum_{p \in \mathbb{Z}^d}(|\xi|^2+\lambda_{ p} )|\widehat{\psi}(\xi, p)|^2+ \tfrac{1}{4} \int_{\xi \in \mathbb{R}} \sum_{|\omega| < \frac{1}{M_n}}\sum_{(p, q, r, s) \in \Gamma_{\omega}} \widehat{\psi}(\xi, q) \widehat{\psi}(\xi, s) \overline{\widehat{\psi}}(\xi, r) \overline{\widehat{\psi}}(\xi, p) .
	\end{align*}

By Proposition \ref{nocasc}, $|g_n(t)|_N=|g_n(\pi (\ln 2) n)|_N$ for $t \in [\pi(\ln 2) n, \pi(\ln 2)(n+1))$. Now define
    \begin{align}
        G_n(t):= e^{-it\Delta_{\R \times \T^d_{\theta}}}g_n(\pi\cdot \ln t) .
    \end{align}
    Then $G_n(t)$ satisfies \eqref{eq effect} on $[2^n, 2^{n+1}]$ and $|G_n(t)|_N=|G_n(2^n)|_N$.
    For $n$ large enough, $K_n^{-1}\|F(2^n)\|_{H^N} \leq \varepsilon 2^{-9n\delta} \leq \varepsilon^2 2^{-2n\delta}$. Therefore, the $\varepsilon_1$ terms from Lemma \ref{asymptotic} dominate and thus 
    \begin{align}
        \norm{F(t) - G_n(t)}_{H^{N-1}} & \lesssim \varepsilon^3 2^{-2n \delta}+ \varepsilon^5 2^{-n\delta}
    \end{align}
for $t \in [2^n, 2^{n+1})$.

Now this implies 
\begin{align}
\norm{G_n (2^{n+1}) - G_{n+1} (2^{n+1})}_{H^{N-1}} \leq \norm{G_n (2^{n+1}) - F(2^{n+1})}_{H^{N-1}} + \norm{F(2^{n+1}) - G_{n+1} (2^{n+1})}_{H^{N-1}}
\end{align}
where the first term is less than $\varepsilon^3 2^{-2n \delta}+ \varepsilon^5 2^{-n\delta}$ due to the estimates above and the second term is less than $\varepsilon 2^{-8n\delta} $ by the definition. Thus for $n$ large enough, $\norm{G_n (2^{n+1}) - G_{n+1} (2^{n+1})}_{H^{N-1}} \lesssim  2^{-n\delta}$.

\begin{claim}\label{claim 3}
If
\begin{align}
\norm{G_n (2^{n+1}) - G_{n+1} (2^{n+1})}_{H^{N-1}} \lesssim \varepsilon^5 2^{-n\delta} ,
\end{align}
then
\begin{align}
\norm{G_n (0) - G_{n+1} (0)}_{H^{N-1}} \lesssim  \varepsilon^52^{-n\delta/2}  .
\end{align}
\end{claim}
The claim follows from Lemma 4.3 in \cite{HPTV15}. 

This implies that $\{ G_n (0)\}_{n=1}^{\infty}$ is Cauchy. Then there exists $G_{\infty}$ such that $G_n (0) \to \infty$ and $\sup_{n} \norm{G_n (0)}_{H^{N-1}} \lesssim \varepsilon^5$. 

Then
\begin{align}
\sup_{n} \sup_{t \in [0, 2^{n+1}]} \norm{G_n(t)}_{H^{N-1}} \lesssim \varepsilon^5. 
\end{align}
Finally, for $T$ large enough,
\begin{align}
    \sup_{t \in [0, \infty)}\|U\|_{H^{N-1}} \leq \sup_{t \in [0, T]}\|U\|_{H^{N}}+\sup_{n>\log T} \sup_{t \in [2^n, 2^{n+1}]} \left(\norm{e^{-it\Delta }U(t) - G_n(t)}_{H^{N-1}} +  \norm{G_n(t)}_{H^{N-1}}\right)<\infty .
\end{align}
    
This completes the proof of Theorem \ref{thm.asym}.

\end{proof}

\begin{remark}\label{rmk.eta}
    For $\eta>0$, the $H^{N-1}$-norm in the statement of Theorem \ref{thm.asym} can be replaced by $H^{N-\eta}$-norm and the proof follows by a redefinition of $M_n$ adapted to $\eta.$ In particular, one would let 
        \begin{align}
            M_n= 2^{(20\gamma\delta n)/\eta}
        \end{align}
    and the proof would follow without introducing a constant depending on $\eta.$ This does not necessarily imply a uniform in $\eta$ bound on the $H^{N-\eta}$-norms of $\{G_j\}_{j=1}$.
\end{remark}

	\begin{theorem}\label{thm.asy}
	Let $d \geq 2$, let $\theta \in \Theta_2$ (recall Definition \ref{def.Theta}), let $N>10 (d+\gamma)$, and let $\mathcal{C}\gg 1$.  There exists a solution to \eqref{eq NLS}, $U(t)$, and two times $0< T_1({\mathcal{C}}) <T_2(\mathcal{C})$ such that 
		\begin{align*}
			\|U(T_2)\|_{H^N} \geq \mathcal{C}\|U(T_1)\|_{H^N},
		\end{align*}
		where $T_2-T_1 \leq \exp(\frac{C_N}{\varepsilon^2}\exp(\mathcal{C}^{\beta}))$ for some $\beta=\beta(N)>0.$
		\end{theorem}

    \begin{proof}[Proof of Theorem \ref{thm.asy}]
    Let $t_1= \exp(\mathcal{C}^{\beta})$ and $t_2= 2t_1$. Consider the solution to \eqref{Rreso} given by Corollary \ref{GrowCor}, $v(0)$, defined so that $\|v(\tau)\|_{H^N}\geq2\mathcal{C} \|v(0)\|_{H^N}$ for some time $\tau\leq t_1$. After a change of the time variable, we have a solution, $G: [0, T_2] \to H^N$, to \eqref{eq effect} (for $M=0$), such that 
        \begin{align}
            \|G(T_1)\|_{H^N}&=\varepsilon<1\\
            \|G\|_{X^+_{T_2}} &\leq 2 \varepsilon\\
            \|G(\exp(\tau\|v(0)\|^2_{H^N}/\varepsilon^2)+T_1)\|_{H^N}&\geq 2\mathcal{C}\varepsilon
        \end{align}
       where $T_1:= \exp(t_1\|v(0)\|^2_{H^N}/\varepsilon^2)$ and $\varepsilon$ is the the same $\varepsilon$ from the statement of Proposition \ref{prop 6.2}. Setting  $e^{-iT_1\Delta}U(T_1):=G(T_1)$, Lemma \ref{asymptotic} implies that 
       \begin{align}
           \|U(\exp(\tau\|v(0)\|^2_{H^N}/\varepsilon^2)+T_1)\|_{H^N}&\geq 2\mathcal{C}\varepsilon - \varepsilon^3T_1^{-\delta} \geq \mathcal{C} \varepsilon .
       \end{align}
       
    \end{proof}

\appendix 

\section{Proof of Proposition \ref{prop 3.1}}\label{app.Proof3.1}
In this appendix we follow the argument of \cite{HPTV15} to prove Proposition \ref{prop 3.1}. We begin as the reference does with the high-frequency estimates.

The different norms used in this section in comparison with the norms used in \cite{HPTV15} does not change essence of the following proofs because the regularity of all norms in this appendix depend on $d$ in such a way as to preserve the relationships that exist in \cite{HPTV15} between the various norms. See, for example, equation \eqref{eq 14}.
\subsection{The High-frequency Estimates}\label{ssec 3.1}

The proof of the following lemma can be found as the proof of Lemma 3.2 in \cite{HPTV15} with the $Z$-norm replaced by the $Z_d$-norm.  

Throughout the appendix we will refer to the symmetric group  $S_3$, which represents the set of all permutations of a set of three elements.
\begin{lemma}\label{lem 3.2}
Assume that $T \geq 1$. The following estimates hold uniformly in $T$:
\begin{align}
& \norm{\sum_{\substack{A,B,C \\ \max (A,B,C) \geq T^{\frac{1}{6}}}} \mathcal{N}^t [Q_AF, Q_BG, Q_C H]}_{Z_d} \lesssim T^{-\frac{7}{6}} \norm{F}_{S} \norm{G}_{S} \norm{H}_{S} , \quad t \geq\tfrac{T}{4} , \\
& \norm{\sum_{\substack{A,B,C \\ \max (A,B,C) \geq T^{\frac{1}{6}}}} \int_{\R} q_T (t) \mathcal{N}^t [Q_AF(t), Q_BG(t), Q_C H(t)] \, dt }_{S} \lesssim T^{-\frac{1}{50}} \norm{F}_{X_T} \norm{G}_{X_T} \norm{H}_{X_T}, \\
& \norm{\sum_{\substack{A,B,C \\ \max (A,B,C) \geq T^{\frac{1}{6}}}} \int_{\R} q_T (t) \mathcal{N}^t [Q_AF(t), Q_BG(t), Q_C H(t)] \, dt }_{S^+} \lesssim T^{-\frac{1}{50}} \norm{F}_{X_T^+} \norm{G}_{X_T^+} \norm{H}_{X_T^+}  . 
\end{align}
\end{lemma}

\subsection{The Fast Oscillations}\label{ssec 3.2}
\begin{lemma}\label{lem 3.3}
For $T \geq 1$, assume that $F, G, H: \R \to S$ satisfy \eqref{eq 3.2} and 
\begin{align}
F = Q_{\leq T^{\frac{1}{6}}} F, \quad G = Q_{\leq T^{\frac{1}{6}}} G, \quad H =Q_{\leq T^{\frac{1}{6}}} H .
\end{align}
Then, for $t \in [\frac{T}{4} , T]$, we can write
\begin{align}
\widetilde{\mathcal{N}}^t [F(t) , G(t) , H(t)] = \mathcal{E}_{1,M}^t + \mathcal{E}_{2,M}^t,
\end{align}
where it holds that, uniformly in $T \geq 1$ and $M>0$,
\begin{align}
T^{1 + 2\delta} \sup_{\frac{T}{4} \leq t \leq T} \norm{\mathcal{E}_{1,M}^t}_{S} \lesssim 1 , \label{eq 15}\\
T^{\frac{1}{10}} \sup_{\frac{T}{4} \leq t \leq T} \norm{\mathcal{E}_{3,M}^t}_{S} \lesssim M, \\
T^{1 + 2\delta} \sup_{\frac{T}{4} \leq t \leq T} \norm{\mathcal{E}_{err, M}^t}_{S} \lesssim 1 .
\end{align}
 where $\mathcal{E}_{2, M}^t= \partial_t\mathcal{E}^t_{3, M}- \mathcal{E}^t_{err}$.  Assuming in addition that \eqref{eq 3.4} holds, we have
\begin{align}
T^{1 + 2\delta} \sup_{\frac{T}{4} \leq t \leq T} \norm{\mathcal{E}_{1,M}^t}_{S^+} \lesssim 1 ,\\
T^{\frac{1}{10}} \sup_{\frac{T}{4} \leq t \leq T} \norm{\mathcal{E}_{3,M}^t}_{S^+} \lesssim M,\\
T^{1 + 2\delta} \sup_{\frac{T}{4} \leq t \leq T} \norm{\mathcal{E}_{err, M}^t}_{S^+} \lesssim 1 .
\end{align}
\end{lemma}

The proof of Lemma \ref{lem 3.3} follows the statement of Lemma \ref{lem 3.6}. We begin with the following remark.

\begin{remark}\label{remark.Eerror} 

Let $M >0$. Recall first we defined $\mathcal{O}_1$ and $\mathcal{O}_2$ in \eqref{eq O}. 
Then recall that we also defined $\mathcal{E}_{3,M}^t$ by 
\begin{align}\label{eq E3}
\mathcal{F} \mathcal{E}_{3,M}^t (\xi , p) : = \sum_{\abs{\omega} \geq \frac{1}{M}} \sum_{(p,q,r,s) \in \Gamma_{\omega}} \frac{e^{it\omega}}{i\omega} \mathcal{O}_2^t [F_q, G_r, H_s] (\xi) 
\end{align}
and we defined $\mathcal{E}^t_{2, M}$ by
    \begin{align*}
        \mathcal{F}  \mathcal{E}_{2,M}^t  (\xi , p)   : =  \sum_{\abs{\omega} \geq \frac{1}{M}}  \sum_{(p,q,r,s) \in \Gamma_{\omega}} e^{it\omega }\mathcal{O}_2^t [F_q, G_r, H_s] (\xi).
    \end{align*}
    By product rule of differentiation we have 
\begin{align}
e^{it \omega} &\mathcal{O}_{2}^t [f,g,h] \\
&= \partial_t \parenthese{\frac{e^{it\omega}}{i \omega} \mathcal{O}_2^t [f,g,h]} - e^{it\omega} (\partial_t \mathcal{O}_2^t) [ f,g,h] - e^{it\omega} \mathcal{O}_2^t [\partial_t f, g,h] - e^{it\omega} \mathcal{O}_2^t [ f, \partial_t g,h]  - e^{it\omega} \mathcal{O}_2^t [ f, g,\partial_t h] ,
\end{align}
which implies
\begin{align}
\partial_t \mathcal{F} \mathcal{E}_{3,M}^t & = \sum_{\abs{\omega} \geq \frac{1}{M}} \sum_{(p,q,r,s) \in \Gamma_{\omega}}  \left\{ e^{it \omega} \mathcal{O}_2^t [F_q,G_r,H_s]  \right. \\
&  \hspace{.3cm}+  \left.  e^{it\omega} (\partial_t \mathcal{O}_2^t) [ F_q,G_r,H_s] +  e^{it\omega} \mathcal{O}_2^t [\partial_t F_q, G_r,H_s] + e^{it\omega} \mathcal{O}_2^t [ F_q, \partial_t G_r,H_s]  +  e^{it\omega} \mathcal{O}_2^t [ F_q, G_r,\partial_t H_s]  \right\} \\
& = :  \mathcal{F}  \mathcal{E}_{2,M}^t + \mathcal{F} \mathcal{E}_{err, M}^t   ,
\end{align}
where we define 
\begin{align}\label{eq 22}
\begin{aligned}
 \mathcal{F} \mathcal{E}_{err, M}^t  &:= \sum_{\abs{\omega} \geq \frac{1}{M}} \sum_{(p,q,r,s) \in \Gamma_{\omega}}   \left\{ e^{it\omega} (\partial_t \mathcal{O}_2^t) [ F_q,G_r,H_s] +  e^{it\omega} \mathcal{O}_2^t [\partial_t F_q, G_r,H_s]   \right\}  \\
 &\hspace{.5cm}+  \sum_{\abs{\omega} \geq \frac{1}{M}} \sum_{(p,q,r,s) \in \Gamma_{\omega}} \left\{e^{it\omega} \mathcal{O}_2^t [ F_q, \partial_t G_r,H_s]  +  e^{it\omega} \mathcal{O}_2^t [ F_q, G_r,\partial_t H_s] \right\} . 
 \end{aligned}
\end{align}

We remark  that the term $\mathcal{E}_{err, M}^t$ behaves nicer than $\mathcal{E}_{2,M}^t $, since $\partial_t F_q$, $\partial_t G_r$ and $\partial_s H$ have better decay than $F$, $G$ and $H$ which ultimately comes from our choice of unknowns as pullbacks of nonlinear solutions by the linear flow.

\end{remark}

\begin{remark}
If for $t \geq T/4$, let 
\begin{align}
f^a = Q_{\leq T^{\frac{1}{6}}} f^a , \quad f^b = Q_{\leq T^{\frac{1}{6}}} f^b, \quad f^c = Q_{\leq T^{\frac{1}{6}}} f^c  ,
\end{align}
then we have the following estimates from \cite{HPTV15} that exploit the dispersion in the $\mathbb{R}$ direction
\begin{align}
\norm{\mathcal{O}_2^t [f^a, f^b ,f^c]}_{L_{\xi}^2} 
& \lesssim (1 + \abs{t})^{-1+\frac{\delta}{100}} \min_{\sigma \in S_3} \norm{f^{\sigma(a)}}_{L_x^2}  \norm{f^{\sigma(b)}}_{Y}  \norm{f^{\sigma(c)}}_{Y}   \label{eq 3.20} ,\\
\norm{\mathcal{O}_1^t [f^a, f^b ,f^c] + \mathcal{O}_2^t[f^a, f^b ,f^c]}_{L_{\xi}^2} & \lesssim (1 + \abs{t})^{-1} \min_{\sigma \in S_3} \norm{f^{\sigma(a)}}_{L_x^2}  \norm{f^{\sigma(b)}}_{Y}  \norm{f^{\sigma(c)}}_{Y} ,
\end{align}
which then implies 
\begin{align}
\norm{\mathcal{O}_1^t [f^a, f^b ,f^c]}_{L_{\xi}^2} 
& \lesssim (1 + \abs{t})^{-1+\frac{\delta}{100}} \min_{\sigma \in S_3} \norm{f^{\sigma(a)}}_{L_x^2}  \norm{f^{\sigma(b)}}_{Y}  \norm{f^{\sigma(c)}}_{Y} .
\end{align}

\end{remark}

The proof of the following lemma follows from the preceding remark using the same argument for the equivalent estimate (Lemma 3.6) in \cite{HPTV15}.

\begin{lemma}\label{lem 3.6}
Assume that $ f^a , f^b, f^c$ satisfy for $t \geq T/4$
\begin{align}
f^a = Q_{A} f^a , \quad f^b = Q_{B} f^b, \quad f^c = Q_{C} f^c  , \quad \max (A, B ,C) \leq T^{\frac{1}{6}} .
\end{align}
Then
\begin{align}
\norm{\mathcal{O}_1^t [f^a, f^b ,f^c]}_{L_{\xi}^2}  \lesssim T^{- \frac{201}{200}} \min_{\sigma \in S_3} \norm{f^{\sigma(a)}}_{L_x^2}  \norm{f^{\sigma(b)}}_{Y}  \norm{f^{\sigma(c)}}_{Y}  .
\end{align}
\end{lemma}

\begin{proof}[Proof of Lemma \ref{lem 3.3}]
One first addresses the term $\mathcal{E}_{1,M}^t$
\begin{align}
\norm{\sum_{\abs{\omega} \geq \frac{1}{M}} \sum_{(p,q,r,s) \in \Gamma_{\omega}} e^{it \omega} \mathcal{O}_1^t [F_q^a, F_r^b, F_s^c]}_{L_{\xi , p}^2} \leq T^{- \frac{201}{200}} \min_{\sigma \in S_3} \norm{F^{\sigma(a)}}_{L_{x,y}^2} \norm{F^{\sigma(b)}}_{S} \norm{F^{\sigma(c)}}_{S} .
\end{align}

By Plancherel,  Lemma \ref{lem 3.6}, 
\begin{align}
 \quad T^{1+ 2 \delta} \sup_{\frac{T}{4} \leq t \leq T} \norm{\mathcal{E}_{1,M}^t}_{H_{x,y}^N} 
& \lesssim  T^{1 + 2 \delta} \max_{\sigma \in S_3} \norm{  \sum_{\abs{\omega} \geq \frac{1}{M}} \sum_{(p,q,r,s) \in \Gamma_{\omega}} \abs{p}^N \norm{ e^{it \omega} \mathcal{O}_1^t [F_q^{\sigma(a)}, F_r^{\sigma(b)}, F_s^{\sigma(c)}]}_{H_{\xi }^N}  }_{\ell_p^2} \\
& \lesssim T^{2 \delta - \frac{201}{200}}  \norm{ \sum_{(p,q,r,s) \in \mathcal{M}} \abs{p}^N  \norm{F_q^a}_{H_x^N} \norm{F_r^b}_{Y} \norm{F_s^c}_{Y}  }_{\ell_p^2}  \\
& \lesssim T^{2 \delta - \frac{201}{200}} \norm{F^a}_{H_{x,y}^N} \norm{F^b}_{S} \norm{F^c}_{S} \\
& \lesssim T^{-\delta - \frac{201}{200}} \norm{F^a}_{X_T} \norm{F^b}_{X_T} \norm{F^c}_{X_T} \\
& \lesssim 1 .
\end{align}
When distributing  derivatives, we could assume that $\abs{q} \geq \abs{r} \geq \abs{s}$, hence we only consider the case when derivatives hit $F_q^a$.

\begin{align}
 \quad T^{1+ 2 \delta} \sup_{\frac{T}{4} \leq t \leq T} \norm{x \mathcal{E}_{1,M}^t}_{L_{x,y}^2} &= T^{1+ 2 \delta} \sup_{\frac{T}{4} \leq t \leq T}  \norm{\mathcal{F}( x \mathcal{E}_{1,M}^t ) }_{L_{\xi , p}^2}  \\
& \lesssim  T^{1+ 2 \delta} \sup_{\frac{T}{4} \leq t \leq T}  \norm{\sum_{\abs{\omega} \geq \frac{1}{M}} \sum_{(p,q,r,s) \in \Gamma_{\omega}} e^{it \omega}  \mathcal{O}_1^t [\partial_{\xi} F_q^{a} , F_r^b, F_s^c]}_{L_{\xi , p}^2}  \\
& \lesssim T^{ 2 \delta -\frac{201}{200}}   \norm{\sum_{(p,q,r,s) \in \mathcal{M}}   \norm{\partial_{\xi} F_q^a }_{L_{\xi}^2}  \norm{ F_r^b }_{Y}  \norm{ F_s^c }_{Y}  }_{\ell_p^2} \\
& \lesssim T^{ 2 \delta -\frac{201}{200}}  \norm{xF^a }_{L_{x,y}^2}  \norm{ F^b }_{S}  \norm{ F^c }_{S} \\
& \lesssim T^{-\delta - \frac{201}{200}} \norm{F^a}_{X_T} \norm{F^b}_{X_T} \norm{F^c}_{X_T} \\
& \lesssim 1 .
\end{align}
The two inequalities above imply the bound on the $S$-norm of $\mathcal{E}^t_{1, M}$ found in \eqref{eq 15}.

The estimate on $\mathcal{E}^t_{3, M}$ follows by first demonstrating an estimate on $\mathcal{O}_2^t$. Define
\begin{align}
\widetilde{\mathcal{O}}_{2, \omega}^t [F, G, H] (\xi , p) := \sum_{ (p,q,r,s) \in \Gamma_{\omega}} \mathcal{O}_2^t [F_q , G_r, H_s] (\xi)  .
\end{align}

Proceeding with a duality argument, let $K \in L_{\xi , p}^2 (\R \times \Z^d)$,
\begin{align}
\inner{K, \widetilde{\mathcal{O}}_{2, \omega}^t [F, G, H] }_{L_{\xi , p}^2 \times L_{\xi , p}^2}  & \leq \sum_{(p,q,r,s) \in \Gamma_{\omega}}  \abs{ \inner{ K_p , \mathcal{O}_2^t[F_q ,G_r, H_s]}_{L_{\xi}^2 \times L_{\xi}^2}} \\
& \lesssim (1 + \abs{t})^{-1+ \delta} \sum_{(p,q,r,s) \in \Gamma_{\omega}} \norm{K_p}_{L_{\xi}^2}    \norm{F_q}_{L_x^2} \norm{G_r}_{Y}  \norm{H_s}_{Y}  . 
\end{align}
Hence 
\begin{align}
& \quad \inner{K, \widetilde{\mathcal{O}}_{2, \omega}^t [F, G, H] }_{L_{\xi , p}^2 \times L_{\xi , p}^2}  \\
& \lesssim  (1 + \abs{t})^{-1 + \delta} \sum_{(p,q,r,s) \in \Gamma_{\omega}} \norm{K_p}_{L_{\xi}^2} \min \bracket{  \norm{F_q}_{L_{x}^2} \norm{G_r}_{Y} \norm{H_s}_{Y} , \norm{F_q}_{Y} \norm{G_r}_{L_{x}^2} \norm{H_s}_{Y}  , \norm{F_q}_{Y} \norm{G_r}_{Y} \norm{H_s}_{L_{x}^2} }  . 
\end{align}

Summing over $\omega$ and using Cauchy-Schwarz in the $p$ variable with \eqref{eq Young}, we get
\begin{align}
\inner{\mathcal{F}^{-1} K, \mathcal{E}_{2,M}^t}_{L_{x,y}^2 \times L_{x,y}^2} & = \inner{K, \sum_{|\omega|\geq \tfrac{1}{M} } e^{it\omega}  \widetilde{\mathcal{O}}_{2, \omega}^t [F^a, F^b, F^c]}_{L_{\xi,p}^2 \times L_{\xi,p}^2}\\
& \lesssim (1 + \abs{t})^{-1 + \delta} \left(\sum_{p \in \mathbb{Z}^d}\norm{K_p}_{L_{\xi}^2}^2 \right)^{\frac{1}{2}}  \min_{\sigma \in S_3}\norm{ \sum_{(p,q,r,s) \in \mathcal{M}}\norm{F^{\sigma(a)}_q}_{L_{x}^2}\norm{F^{\sigma(b)}_r}_{Y} \norm{F^{\sigma(c)}_s}_{Y} }_{\ell^2_p}\\
& \lesssim (1 + \abs{t})^{-1 + \delta}  \norm{K}_{L^2_{\xi, p}}  \min_{\sigma \in S_3} \left(\sum_{q \in \mathbb{Z}^d} \norm{F^{\sigma(a)}_q}^2_{L_{x}^2}\right)^{\frac{1}{2}}\left(\sum_{r \in \mathbb{Z}^d}\norm{F^{\sigma(b)}_r}_{Y}\right) \left(\sum_{s \in \mathbb{Z}^d}\norm{F^{\sigma(c)}_s}_{Y} \right)\\
& \lesssim (1 + \abs{t})^{-1 + \delta}  \norm{K}_{L_{\xi, p}^2} \min_{\sigma \in S_3} \norm{F^{\sigma(a)}}_{L_{x,y}^2}  \norm{F^{\sigma(b)}}_{S} \norm{F^{\sigma(c)}}_{S}  .
\end{align}
where the last step follows from inequality \eqref{defn Y}. Hence
\begin{align}
 \norm{\sum_{\omega } e^{it\omega}  \widetilde{\mathcal{O}}_{2, \omega}^t [F^a, F^b, F^c] }_{L_{\xi , p}^2} &\lesssim (1 + \abs{t})^{-1 + \delta}  \min_{\sigma \in S_3} \norm{F^{\sigma(a)}}_{L_{x,y}^2}  \norm{F^{\sigma(b)}}_{S} \norm{F^{\sigma(c)}}_{S}.
\end{align}

We want to show that
\begin{align}
T^{\frac{1}{10}} \sup_{\frac{T}{4} \leq t \leq T} \norm{\mathcal{E}_{3,M}^t}_{S} \lesssim M   .
\end{align}

Again, we proceed with the same duality argument. For $K \in L_{\xi , p}^2 (\R \times \Z^d)$,
\begin{align}
\inner{\mathcal{F}^{-1} K, \mathcal{E}_{3,M}^t}_{L_{x,y}^2 \times L_{x,y}^2} & = \inner{ K , \mathcal{F} \mathcal{E}_{3,M}^t}_{L_{\xi , p}^2 \times L_{\xi , p}^2} \\
& \leq \sum_{\abs{\omega} \geq \frac{1}{M}} \abs{ \inner{  K_p , \sum_{(p,q,r,s) \in \Gamma_{\omega}} \frac{e^{it\omega}}{i\omega} \mathcal{O}_2^t [F_q, G_r, H_s] }} \\
& \lesssim \sum_{\abs{\omega} \geq \frac{1}{M}} \sum_{(p,q,r,s) \in \Gamma_{\omega}} \frac{1}{\abs{\omega}} \abs{ \inner{ K_p , \mathcal{O}_2^t [F_q , G_r , H_s]}} \\
& \lesssim M (1 + \abs{t})^{-1 + \delta} \norm{\mathcal{F}^{-1} K}_{L_{x,y}^2} \norm{F}_{L_{x,y}^2} \norm{G}_{S} \norm{H}_{S} .
\end{align}
Hence
\begin{align}
& \quad \inner{\mathcal{F}^{-1} K, \mathcal{E}_{3,M}^t}_{L_{x,y}^2 \times L_{x,y}^2} \\
& \lesssim M (1 + \abs{t})^{-1 + \delta} \norm{\mathcal{F}^{-1} K}_{L_{x,y}^2} \min \bracket{  \norm{F}_{L_{x,y}^2} \norm{G}_{S} \norm{H}_{S} , \norm{F}_{S} \norm{G}_{L_{x,y}^2} \norm{H}_{S}  , \norm{F}_{S} \norm{G}_{S} \norm{H}_{L_{x,y}^2} }  .
\end{align}
This completes the bound for $\mathcal{E}^t_{3, M}$.

We refer to remark \ref{remark.Eerror} for the conclusion of the estimate for $\mathcal{E}_{err, M}^t$. The $S^+$-norm estimates follow similarly.

\end{proof}

\subsection{The Resonant Level Set}\label{ssec 3.3}

We now want to show that the resonant part of the $\mathcal{N}^t$ nonlinearity, $\Pi^t_M$, converges to the nonlinearity associated to the effective equation, $\tfrac{\pi}{t} \mathcal{R}_M^t$, as time approaches infinity.
\begin{align}
\mathcal{F} \Pi_M^t [F,G,H] (\xi , p)  = \sum_{\abs{\omega} < \frac{1}{M}} \sum_{(p,q,r,s) \in \Gamma_{\omega}} e^{it\omega} ( \mathcal{I}^t [F_{q} , G_r , H_s])^{\wedge} (\xi) .
\end{align}

We recall the $Z_{t, d}$-norm defined by  Definition \ref{defn Z_td}  in the following lemma
\begin{lemma}\label{lem 3.7}
Let $t \geq 1$. Uniformly in $M\geq 0$, it holds that
\begin{align}
\norm{\Pi_M^t [F^a , F^b , F^c]}_{S} & \lesssim (1 + \abs{t})^{-1} \sum_{\sigma \in S_3} \norm{F^{\sigma (a)}}_{Z_{t,d}} \norm{F^{\sigma (b)}}_{Z_{t,d}}  \norm{F^{\sigma (c)}}_{S} , \label{eq 19}\\
\norm{\Pi_M^t [F^a , F^b , F^c]}_{S^+} & \lesssim (1 + \abs{t})^{-1} \sum_{\sigma \in S_3} \norm{F^{\sigma (a)}}_{Z_{t,d}} \norm{F^{\sigma (b)}}_{Z_{t,d}}  \norm{F^{\sigma (c)}}_{S^+}  \\
& \quad + (1 + \abs{t})^{-1+ 2 \delta} \sum_{\sigma \in S_3} \norm{F^{\sigma (a)}}_{Z_{t,d}} \norm{F^{\sigma (b)}}_{S}  \norm{F^{\sigma (c)}}_{S}  \label{eq 20}.
\end{align}
In addition,
\begin{align}
\norm{\Pi_M^t [F,G,H] - \frac{\pi}{t} \mathcal{R}_M^t [F,G,H]}_{Z_d} & \lesssim (1 + \abs{t})^{-1-20\delta} \norm{F}_{S} \norm{G}_{S} \norm{H}_{S} ,\\
\norm{\Pi_M^t [F,G,H] - \frac{\pi}{t} \mathcal{R}_M^t [F,G,H]}_{S} & \lesssim (1 + \abs{t})^{-1-20\delta} \norm{F}_{S^+} \norm{G}_{S^+} \norm{H}_{S^+} .
\end{align}
\end{lemma}

\begin{remark}\label{rmk Replacement}
\begin{enumerate}
\item 
The goal of Lemma \ref{lem 3.7} is to show 
\begin{align}
\Pi_M^t = \frac{\pi}{t} \mathcal{R}_M^t  + O (\abs{t}^{-1 - 2 \delta}) .
\end{align}

\item
It is clear from the proof of Lemma \ref{lem 3.7} that \eqref{eq 19} and \eqref{eq 20} also hold if one replaces $\Pi_M^t [F^a,F^b,F^c]$ by $\frac{\pi}{t} \mathcal{R}_M^t [F^a,F^b,F^c]$.
\end{enumerate}

\end{remark}

We recall Lemma 3.10 from \cite{HPTV15} in the exact same presentation.

\begin{lemma}\label{lem 3.10}
Assume that
\begin{align}
f(x) = \phi (s^{-\frac{1}{4}} x ) f(x) , \quad g(x) = \phi (s^{-\frac{1}{4}} x ) g(x) , \quad h(x) = \phi (s^{-\frac{1}{4}} x ) h(x), 
\end{align}
and that $s \geq 1$. It holds that 
\begin{align}
\abs{\int_{\R^2} e^{2 i s \eta \kappa} \widehat{f} (\xi - \eta) \overline{\widehat{g}} (\xi - \eta - \gamma) \widehat{h} (\xi  - \kappa) \, d \eta d \kappa - \frac{\pi}{s} \widehat{f} (\xi ) \overline{\widehat{g}} (\xi) \widehat{h} (\xi) } \lesssim s^{- \frac{9}{8}} \norm{f}_{L_x^2} \norm{g}_{L_x^2} \norm{h}_{L_x^2} .
\end{align}
In fact, for $\theta \in \N$,
\begin{align}
&\quad \abs{\xi}^{\theta}  \abs{\int_{\R^2} e^{2 i s \eta \kappa} \widehat{f^a} (\xi - \eta) \overline{\widehat{f^b}} (\xi - \eta - \gamma) \widehat{f^c} (\xi  - \kappa) \, d \eta d \kappa - \frac{\pi}{s} \widehat{f^a} (\xi ) \overline{\widehat{f^b}} (\xi) \widehat{f^c} (\xi) } \\
&\lesssim s^{- \frac{9}{8}} \min_{\sigma \in S_3} \norm{f^{\sigma(a)}}_{H_x^{\theta}} \norm{f^{\sigma(b)}}_{L_x^2}  \norm{f^{\sigma(c)}}_{L_x^2} .
\end{align} 
\end{lemma}

The following lemma will act as the analogue of the Strichartz estimate lemma, Lemma 7.1 from \cite{HPTV15}:

\begin{lemma}[Weaker version]\label{lem 7.1}
Consider the function, $R^t_M: [h^{d/2+}(\mathbb{Z}^d)]^3 \to \ell^2(\mathbb{Z}^d)$, defined by 
    \begin{align}
        R^t_M[a, b, c](p) := \sum_{\abs{\omega} < \frac{1}{M}} e^{it \omega}\sum_{(p,q,r,s) \in \Gamma_{\omega}} a(q)\overline{b}(r) c(s).
    \end{align}
If $a^{(i)} \in h^{d/2+}$ for $i =1, 2, 3$,
then
\begin{align}
\norm{R^t_M  [a^{(1)} , a^{(2)} , a^{(3)}]}_{\ell^2} \lesssim_{d} \min_{\sigma \in S_3} \norm{a^{\sigma(1)}}_{\ell^2} \norm{a^{\sigma(2)}}_{h^{\frac{d}{2}+}} \norm{a^{\sigma(3)}}_{h^{\frac{d}{2}+}} ,\\
\norm{R^t_M  [a^{(1)} , a^{(2)} , a^{(3)}]}_{h^N} \lesssim_{d,N} \max_{\sigma \in S_3} \norm{a^{\sigma(1)}}_{h^N} \norm{a^{\sigma(2)}}_{h^{\frac{d}{2}+}} \norm{a^{\sigma(3)}}_{h^{\frac{d}{2}+}} .
\end{align}
Uniformly in $M\geq 0.$
\end{lemma}

\begin{remark}[The significance of Lemma \ref{lem 7.1}]
This follows from basic Young's inequality and Sobolev embedding.

We note that this estimate is worse than the analogous estimate of \cite{HPTV15} derived from sophisticated Strichartz  estimates. It is well-known that when the torus is re-scaled by an irrational vector, Strichartz estimates are not as easily attainable. Even though the sharp Strichartz estimates hold in the case that we are considering (see Bourgain and Demeter \cite{BD15},  Kilip and Visan \cite{KV16} and Herr, Tataru and Tzvetkov \cite{HTT14}), the argument used in \cite{HPTV15} is not easily reconstituted in this setting since frequency localized solutions to the linear Schr\"odinger equation are not necessarily time-periodic. One can achieve less derivative loss on the right-hand sides of the inequalities that appear in Lemma \ref{lem 7.1}, but we have no use in our argument for the best possible bound. In \citep{HPTV15}, they used the better version of Lemma \ref{lem 7.1} in two ways:
\begin{enumerate}
\item
It is used to bound $\Pi_M^t$;
\item
Most importantly, the $h^1$ norm can be shown to be a constant of motion with respect to the resonant equation. This means that an $h^1$ upper bound is particularly important to the description of asymptotic behavior of solutions. We do not need this because for our effective equation, all $h^s$ norms are bounded (for special solutions). This further allows the dimension of are torus to be arbitrarily high.
\end{enumerate}
\end{remark}

\begin{lemma}\label{lem 7.3}
Assume that $N \geq 10d$. Then
\begin{align}
\sup_{x \in \R} \sum_{p \in \Z^d} (1 + \abs{p}^2)^{\frac{d}{2}+} \abs{e^{i t \partial_{xx}} F_p(x)}^2 \lesssim \inner{t}^{-1} \parenthese{\norm{F}_{Z_d}^2 + \inner{t}^{-\frac{1}{10}} (\norm{xF}_{L^2_{x, y}}^2 + \norm{F}_{H^N_{x, y}}^2) }
\end{align}
where the $Z_d$ norm is defined in Definition \ref{defn Z_d}.
\end{lemma}

\begin{proof}[Proof of Lemma \ref{lem 7.3}]
Without loss of generality, suppose $t \geq 1$. 

We recall the following asymptotic approximation of $e^{it \partial_{xx}} f$ in \cite{HPTV15}:
\begin{align}\label{eq asy}
\abs{e^{it \partial_{xx}} f(x) - c \frac{e^{-i \frac{x^2}{4t}}}{\sqrt{t}} \widehat{f} (- \frac{x}{2t}) } \lesssim t^{- \frac{3}{4}} \norm{x f}_{L^2_x} .
\end{align}

Now using \eqref{eq asy} and triangle inequality, we write
\begin{align}
t \cdot \sum_{p \in \Z^d , \abs{p} \leq t^{\frac{1}{5d}}}  (1 + \abs{p}^2)^{\frac{d}{2}+} \abs{e^{it\partial_{xx}} F_p (x)}^2 
& \lesssim \sum_{p \in \Z^d} (1 + \abs{p}^2)^{\frac{d}{2}+} \abs{\widehat{F}_p (- \frac{x}{2t})}^2 + t^{- \frac{1}{2}} \sum_{p \in \Z^d , \abs{p} \leq t^{\frac{1}{5d}}} (1 + \abs{p}^2)^{\frac{d}{2}+} \norm{xF_p}_{L^2_x}^2 \\
& \lesssim \sum_{p \in \Z^d} (1 + \abs{p}^2)^{\frac{d}{2}+} \abs{\widehat{F}_p (- \frac{x}{2t})}^2 +  t^{- \frac{1}{2}} \sum_{p \in \Z^d , \abs{p} \leq t^{\frac{1}{5d}}}  t^{\frac{d+}{5d}}  \norm{xF_p}_{L^2_x}^2 \\
& \leq \sum_{p \in \Z^d} (1 + \abs{p}^2)^{\frac{d}{2}+} \abs{\widehat{F}_p (- \frac{x}{2t})}^2 + t^{- \frac{1}{2}} t^{\frac{1}{5}+}  \norm{xF}_{L_{x, y}^2}^2.
\end{align}
Observe that $\sup_{x, t} \abs{\widehat{F}_p (- \frac{x}{2t})}^2 \leq \sup_{\xi} \abs{\widehat{F}_p ( \xi)}^2$, thus $\sum_{p \in \Z^d} (1 + \abs{p}^2)^{\frac{d}{2}+} \abs{\widehat{F}_p (- \frac{x}{2t})}^2 \leq \norm{F}_{Z_d}^2$. This inequality dictates the form of the $Z_d$ norm. Furthermore
\begin{align}
t \sum_{\abs{p} \geq t^{\frac{1}{5d}}} (1 + \abs{p}^2)^{\frac{d}{2}+} \abs{e^{it\partial_{xx}} F_p (x)}^2 & \lesssim t^{1- \frac{10d}{5d}} \sum (1 + \abs{p}^2)^{5d+\frac{d}{2} + } \norm{F_p}_{H^1}^2  \leq t^{-1} \norm{F}_{H^N}^2 .
\end{align}
This final estimate establishes the strongest lower bound for any feasible choice of $N$ in terms of $d$.

\end{proof}

\begin{remark}\label{rmk.5/2}
    By reducing the decay of $\langle t \rangle^{-1/10}$ one can achieve the $N$ condition
        \begin{align}
            N>\tfrac{5}{2}d .
        \end{align}
    We found no use for this improved bound in this manuscript.
\end{remark}

\begin{proof}[Proof of Lemma \ref{lem 3.7}]
Using the definition of $\Pi_M^t$ in \eqref{eq Pi}, we have
\begin{align}
\norm{\Pi_M^t [F^1 , F^2 , F^3]}_{L_{x,y}^2} & = \norm{\sum_{\abs{\omega} \leq \frac{1}{M}}  \sum_{(p,q,r,s) \in \Gamma_{\omega}} e^{i \omega t} \mathcal{I}^t [F_q^{(1)} , F_r^{(2)}, F_s^{(3)}] (x)}_{\ell_p^2 L_x^2} \\
& \leq \norm{\sum_{\abs{\omega} \leq \frac{1}{M}}  \sum_{(p,q,r,s) \in \Gamma_{\omega}} e^{i\omega t} e^{-it \partial_{xx}} (e^{it \partial_{xx}} F_q^{(1)} (x)) (\overline{e^{it \partial_{xx}} F_r^{(2)} (x)}) (e^{it \partial_{xx}} F_s^{(3)} (x))  }_{\ell_p^2 L_x^2} . \label{eq 2}
\end{align}
Let 
\begin{align}
a_q^{(1)} (x) : = e^{it \partial_{xx}} F_q^{(1)} (x), \quad a_r^{(2)} (x) : = e^{it \partial_{xx}} F_r^{(2)} (x),  \quad a_s^{(3)} (x) : = e^{it \partial_{xx}} F_s^{(3)} (x) .
\end{align}
Then using Lemma \ref{lem 7.1}
\begin{align}
\eqref{eq 2} & = \norm{R_M (a^{(1)} , a^{(2)} , a^{(3)}) (x)}_{\ell_p^2 L_x^2} \\
& \leq \norm{\norm{a^{(1)} (x)}_{\ell_p^2}}_{L_x^2} \norm{ \norm{a^{(2)} (x)}_{h_p^{\frac{d}{2}+}}}_{L_x^{\infty}}  \norm{\norm{a^{(3)} (x)}_{h_p^{\frac{d}{2}+}} }_{L_x^{\infty}} .
\end{align}

Now we want to use Lemma \ref{lem 7.3} to bound $ \norm{\norm{a^{(j)} (x)}_{h_p^{\frac{d}{2}+}} }_{L_x^{\infty}}$, $j=2,3$. 
By Lemma \ref{lem 7.3},
\begin{align}
 \norm{\norm{a^{(j)} (x)}_{h_p^{\frac{d}{2}+}} }_{L_x^{\infty}}  \leq \inner{t}^{-\frac{1}{2}} \parenthese{\norm{F^{(j)}}_{Z_d} + \inner{t}^{-\frac{1}{20}} (\norm{xF^{(j)}}_{L_x^2} + \norm{F^{(j)}}_{H_x^N}) }  .
\end{align}
Then
\begin{align}
 \norm{\norm{a^{(j)} (x)}_{h_x^{\frac{d}{2}+}} }_{L_x^{\infty}} \lesssim \inner{t}^{- \frac{1}{2}} \norm{F^{(j)}}_{Z_{t,d}} ,
\end{align}
where the $Z_{t,d}$ norm is defined in Definition \ref{defn Z_td}.

Finally,
\begin{align}
\norm{\Pi_M^t [F^{(1)} , F^{(2)} , F^{(3)}]}_{L_{x,y}^2} \lesssim \inner{t}^{-1} \min_{\sigma \in S_3} \norm{F^{\sigma(1)}}_{L_{x,y}^2} \norm{F^{\sigma(2)}}_{Z_{t,d}} \norm{F^{\sigma(3)}}_{Z_{t,d}} .
\end{align}
This basically proves the first part of Lemma \ref{lem 3.7}. 

To prove the `in addition' estimates, decompose 
\begin{align}
F = F_c + F_f , \quad G = G_c + G_f , \quad H = H_c + H_f ,
\end{align}
where
\begin{align}
F_c (x,y) : = \phi (t^{- \frac{1}{4}} x) F(x,y) , \\
G_c (x,y) : = \phi (t^{- \frac{1}{4}} x) G(x,y) ,\\
H_c (x,y) : = \phi (t^{- \frac{1}{4}} x) H(x,y) .
\end{align}

\begin{claim}\label{claim diff}
The following estimates hold uniformly in $M\geq0$.
\begin{align}
\norm{\Pi_M^t [F,G,H] - \Pi_M^t [F_c, G_c, H_c]}_{Z_d} + \frac{1}{t} \norm{\mathcal{R}_M^t [F,G,H] - \mathcal{R}_M^t [F_c , G_c, H_c]}_{Z_d} \lesssim \inner{t}^{- \frac{21}{20}} \norm{F}_{S} \norm{G}_{S} \norm{H}_{S} .
\end{align}
and 
\begin{align}
\norm{\Pi_M^t [F,G,H] - \Pi_M^t [F_c, G_c, H_c]}_{S} + \frac{1}{t} \norm{\mathcal{R}_M^t [F,G,H] - \mathcal{R}_M^t [F_c , G_c, H_c]}_{S} \lesssim \inner{t}^{- \frac{21}{20}} \norm{F}_{S^+} \norm{G}_{S^+} \norm{H}_{S^+} .
\end{align}
\end{claim}

\begin{proof}[Proof of Claim \ref{claim diff}]
Let $\widetilde{G} = G_c $ or $G_f$, and $\widetilde{H} = H_c $ or $H_f$. Recall that 
\begin{align}
    \norm{\frac{\pi}{t} \mathcal{R}_M^t [F_f , \widetilde{G}, \widetilde{H}]}_{Z_d} \leq \norm{\frac{\pi}{t} \mathcal{R}_M^t [F_f , \widetilde{G}, \widetilde{H}]}_{L^2_{x, y}}^{\frac{1}{4}}\norm{\frac{\pi}{t} \mathcal{R}_M^t [F_f , \widetilde{G}, \widetilde{H}]}_{S}^{\frac{3}{4}}
\end{align}
and similarly 
\begin{align}
\norm{\Pi_M^t }_{Z_d} \lesssim \norm{\Pi_M^t }_{L^2_{x, y}}^{\frac{1}{4}}\norm{\Pi_M^t }_{S}^{\frac{3}{4}} .
\end{align}
By a basic trilinear estimate  
\begin{align}
\norm{\frac{\pi}{t} \mathcal{R}_M^t [F_f , \widetilde{G}, \widetilde{H}]}_{S} \lesssim \inner{t}^{-1} \norm{F}_{S} \norm{G}_{S} \norm{H}_{S}.
\end{align}
By the first estimate in Lemma \ref{lem 3.7}
\begin{align}
\norm{\Pi_M^t [F_f , \widetilde{G}, \widetilde{H}]}_{S} \lesssim \inner{t}^{-1} \norm{F}_{S} \norm{G}_{S} \norm{H}_{S} .
\end{align}
Similarly,
\begin{align}
\norm{\frac{\pi}{t} \mathcal{R}_M^t [F_f , \widetilde{G}, \widetilde{H}] }_{L^2} + \norm{\Pi_M^t [F_f , \widetilde{G}, \widetilde{H}]}_{L^2} & \lesssim \inner{t}^{-1} \norm{F_f}_{L^2_{x, y}} \norm{\widetilde{G}}_{S} \norm{\widetilde{H}}_{S} .
\end{align}
By Bernstein's inequality on the $F_f$ term,
\begin{align}
    \inner{t}^{-1} \norm{F_f}_{L^2_{x, y}} \norm{\widetilde{G}}_{S} \norm{\widetilde{H}}_{S}\lesssim \inner{t}^{-\frac{5}{4}} \norm{F}_{S} \norm{G}_{S} \norm{H}_{S} .
\end{align}
Thus 
    \begin{align}
        \norm{\frac{\pi}{t} \mathcal{R}_M^t [F_f , \widetilde{G}, \widetilde{H}] }_{L^2} + \norm{\Pi_M^t [F_f , \widetilde{G}, \widetilde{H}]}_{L^2} \lesssim \inner{t}^{-\frac{5}{4}} \norm{F}_{S} \norm{G}_{S} \norm{H}_{S}
    \end{align}
    and
    \begin{align}
        \norm{\frac{\pi}{t} \mathcal{R}_M^t [F_f , \widetilde{G}, \widetilde{H}] }_{S} + \norm{\Pi_M^t [F_f , \widetilde{G}, \widetilde{H}]}_{S} \lesssim \inner{t}^{-\frac{5}{4}} \norm{F}_{S^+} \norm{G}_{S^+} \norm{H}_{S^+}
    \end{align}
\end{proof}

Claim \ref{claim diff} implies that is suffices to show
\begin{align}
\norm{\Pi_M^t [F_c, G_c, H_c] - \frac{\pi}{t} \mathcal{R}_M^t [F_c , G_c, H_c] }_{Z_d} \lesssim \inner{t}^{- \frac{15}{14}} \norm{F}_{S} \norm{G}_{S} \norm{H}_{S} ,
\end{align}
and
\begin{align}
\norm{\Pi_M^t [F_c, G_c, H_c] - \frac{\pi}{t} \mathcal{R}_M^t [F_c , G_c, H_c] }_{S} \lesssim \inner{t}^{- \frac{15}{14}} \norm{F}_{S^+} \norm{G}_{S^+} \norm{H}_{S^+} .
\end{align}
This follows from Lemma \ref{lem 3.10}. Note that although the gauge transformation terms appear in $\Pi^t_M$ and $\mathcal{R}^t_M$ in the irrational torus case. However, term-by-term, the nonlinearities have the same unit modulus coefficients appearing on the same monomials.

The rest of the proof follows exactly as it does in \cite{HPTV15} as the proof of Lemma 3.7.
\end{proof}

\begin{proof}[Proof of Proposition \ref{prop 3.1}]
The proof follows exactly as in \cite{HPTV15} with the estimates given in this section.
\end{proof}

\section{Proof of Proposition \ref{nocascprop}} \label{app.nocasc}

We include the proof of Proposition \ref{nocascprop} in this section.
\begin{proof}

First, by Corollary \ref{nogrowth}, $v_{\xi, p}=0$, for $p \in \mathbb{Z}^d\setminus B_{K/3}$. For $p \in B_{K/3}$,

	\begin{align*} 
	 	\tfrac{d}{dt}v_{\xi, p}
	 	&= i(|\xi|^2+\lambda_{ p}) v_{\xi, p} +i\sum_{(p, q, r, s) \in \Gamma_0^K} v_{\xi, q}v_{\xi, s}\overline{v}_{\xi, r} +i\sum_{|\omega|< \frac{1}{M}}\sum_{ (p, q, r, s) \in \Gamma_{\omega} \atop \max(\|p\|, \|q\|, \|r\|, \|s\|)>K} v_{\xi, q}v_{\xi, s}\overline{v}_{\xi, r} \\
	 	&=i(|\xi|^2+\lambda_{ p}) v_k +i\sum_{(p, q, r, s) \in \Gamma_0^{K/3}} v_{\xi, q}v_{\xi, s}\overline{v}_{\xi, r} .
	 \end{align*}
Using $v_{\xi, p}=0$  for $\|p\|>K/3$, we can simplify the following expression
	\begin{align*}
		\tfrac{d}{dt} |v(t)|_{k}^2 &= \tfrac{d}{dt} \int_{\xi \in \mathbb{R}}\sum_{p \in \mathbb{Z}^d} \left[\sum_{i=1}^d (1+|p_i|^2)^k + (1+|\xi|^2)^k \right]|v_{\xi, p}|^2\\
		&=\tfrac{d}{dt}\int_{\xi \in \mathbb{R}}\sum_{p \in B_{K/3}} \left[\sum_{i=1}^d (1+|p_i|^2)^k + (1+|\xi|^2)^k \right]|v_{\xi, p}|^2. 
	\end{align*}
	In order to simplify the following computations, let $ [ p]_k:= \sum_{i=1}^d (1+|p_i|^2)^k $. We first split into two terms. The argument for the first term will imply our desired estimate for the second term
	\begin{align*}
		\int_{\xi \in \mathbb{R}}\tfrac{d}{dt}& \sum_{p \in B_{K/3}}\left[[ p]_k + (1+|\xi|^2)^k \right]|v_{\xi, p}|^2\\
		&=\int_{\xi \in \mathbb{R}}\tfrac{d}{dt} \sum_{p \in B_{K/3}} [ p]_k|v_{\xi, p}|^2+\int_{\xi \in \mathbb{R}}(1+|\xi|^2)^k\tfrac{d}{dt}\sum_{p \in B_{K/3}} |v_{\xi, p}|^2\\
		&=I+II .
	\end{align*}
	For term I, we would like to show the following claim: 
	\begin{claim}\label{claim constant} 
		\begin{align*}
		\tfrac{d}{dt} \sum_{p \in B_{K/3}} [ p]_k|v_{\xi, p}|^2=0
	\end{align*}
	for any $k\geq 0$.
	\end{claim}
	This would imply that both $I$ and $II$ equal zero. To that end, we expand the derivative of the product $|v_{\xi, p}|^2$ and insert the vector field:
	\begin{align*}
		\int_{\xi \in \mathbb{R}}\tfrac{d}{dt} &\sum_{p \in B_{K/3}} [ p]_k|v_{\xi, p}|^2\\
		&= \int_{\xi \in \mathbb{R}} \sum_{p \in B_{K/3}} [ p]_k(\tfrac{d}{dt} v_{\xi, p}\overline{v}_{\xi, p} + v_{\xi, p} \tfrac{d}{dt}\overline{y}_{\xi, p})\\
		&=\int_{\xi \in \mathbb{R}} \sum_{p \in B_{K/3}}[ p]_k\left(i(|\xi|^2+\lambda_p) v_{\xi, p} +i \sum_{(p, q, r, s) \in \Gamma_0^{K/3}}   v_{\xi, q}v_{\xi, s}\overline{v}_{\xi, r}\right) \overline{v}_{\xi, p} \\
		&\hspace{2cm}+ [ p]_k v_{\xi, p}\left( \overline{i(|\xi|^2+\lambda_{p}) v_{\xi, p}+i \sum_{(p, q, r, s) \in \Gamma_0^{K/3}}   v_{\xi, q}v_{\xi, s}\overline{v}_{\xi, r}\overline{v}_{\xi, p}} \right) .
	\end{align*}
 Simplifying, we obtain
	\begin{align*}
	&\sum_{p \in B_{K/3}} [ p]_k\left(i(|\xi|^2+\lambda_{ p})v_{\xi, p} +i \sum_{(p, q, r, s) \in \Gamma^{K/3}_0}   v_{\xi, q}v_{\xi, s}\overline{v}_{\xi, r}\right) \overline{v}_{\xi, p}\\
	&\hspace{4cm}+ [ p]_k v_{\xi, p}\left( \overline{i(|\xi|^2+\lambda_{p}) v_{\xi, p}+i \sum_{(p, q, r, s) \in \Gamma_0^{K/3}}   v_{\xi, q}v_{\xi, s}\overline{v}_{\xi, r}\overline{v}_{\xi, p} }\right) \\
		&=\sum_{p \in B_{K/3}}i[ p]_k((|\xi|^2+\lambda_{p})^2|v_{\xi, p}|^2-(|\xi|^2+\lambda_{p})^2|v_{\xi, p}|^2) \\
		&\hspace{2cm}+\sum_{p \in B_{K/3}}   i[ p]_k\sum_{(p, q, r, s) \in \Gamma_0^{K/3}}   v_{\xi, q}v_{\xi, s}\overline{v}_{\xi, r}\overline{v}_{\xi, p}\\
			&\hspace{4cm} +\sum_{p \in B_{K/3}}   -i[ p]_k\sum_{(p, q, r, s) \in \Gamma_0^{K/3}}   \overline{v}_{\xi, q}\overline{v}_{\xi, s}v_{\xi, r}v_{\xi, p} .
	\end{align*}
	Cancelling the real parts of this sum of complex numbers leads to the following identity
	\begin{align*}
	\int_{\xi \in \mathbb{R}}&\tfrac{d}{dt} \sum_{p \in B_{K/3}} [ p]_k|v_{\xi, p}|^2=2\im\left(\sum_{(p, q, r, s) \in \Gamma_0^{K/3}}   [ p]_k v_{\xi, q}v_{\xi, s}\overline{v}_{\xi, r}\overline{v}_{\xi, p} \right) .
	\end{align*}

Now we use the resonance condition to split the sum into $d$ one-dimensional sums. First, we use the linearity of $[ p]_k$ to split the sum into $d$ separate components
	\begin{align*}
		&2\im\left(\sum_{(p, q, r, s) \in \Gamma_0^{K/3}}   [ p]_k v_{\xi, q}v_{\xi, s}\overline{v}_{\xi, r}\overline{v}_{\xi, p} \right)=\sum_{i=1}^d 2\im\left(\sum_{(p, q, r, s) \in \Gamma_0^{K/3}}   (1+|p_i|^2)^k v_{\xi, q}v_{\xi, s}\overline{v}_{\xi, r}\overline{v}_{\xi, p} \right) .
	\end{align*}
	Consider the $i$th term:
	\begin{align*}
		2\im\left(\sum_{(p, q, r, s) \in \Gamma_0^{K/3}}   (1+|p_i|^2)^k v_{\xi, q}v_{\xi, s}\overline{v}_{\xi, r}\overline{v}_{\xi, p} \right) .
	\end{align*}
	By Lemma \ref{char}, for each $j \in \{1, ..., d\}$, we know that ($p_j=q_j$ and $r_j=s_j$), or ($p_j=s_j$ and $r_j=q_j$). For $E\subset \{ 1,..., d\}$, let $\Gamma_0^{K/3}(E)$ be the subset of $\Gamma_0^{K/3}$ such that 
	    \begin{align}
	        \begin{array}{ll}
	            p_j=q_j, r_j=s_j & j \in E ,\\
	            p_j=s_j, r_j=q_j & j \not\in E .
	        \end{array}
	    \end{align}
	  Then
	  \begin{align}
	      2\im&\left(\sum_{(p, q, r, s) \in \Gamma_0^{K/3}}   (1+|p_i|^2)^k v_{\xi, q}v_{\xi, s}\overline{v}_{\xi, r}\overline{v}_{\xi, p} \right)=\sum_{E \subset \{1, ..., d\}}2\im\left(\sum_{(p, q, r, s) \in \Gamma_0^{K/3}(E)}   (1+|p_i|^2)^k v_{\xi, q}v_{\xi, s}\overline{v}_{\xi, r}\overline{v}_{\xi, p} \right).
	  \end{align}
	  Furthermore, if we let $p_E \in \mathbb{Z}^d$ be defined by $(p_E)_j = p_j$ for $j \in E$ and $(p_E)_j= 0$ for $j \not\in E$, then
	\begin{align}
	    2\im&\left(\sum_{(p, q, r, s) \in \Gamma_0^{K/3}(E)}   (1+|p_i|^2)^k v_{\xi, q}v_{\xi, s}\overline{v}_{\xi, r}\overline{v}_{\xi, p} \right)=2\im\left(\sum_{|p_j|, |r_j|\leq K/3 }   (1+|p_i|^2)^k v_{\xi, p_E+r_{E^c}}v_{\xi, p_{E^c}+r_E}\overline{v}_{\xi, r}\overline{v}_{\xi, p} \right) .
	\end{align}
	Without loss of generality, let $i \in E$ (otherwise, reverse the definition of $E$), another decomposition yields
	\begin{align}
	2\im&\left(\sum_{|p_j|, |r_j|\leq K/3 }   (1+|p_i|^2)^k v_{\xi, p_E+r_{E^c}}v_{\xi, p_{E^c}+r_E}\overline{v}_{\xi, r}\overline{v}_{\xi, p} \right)\\
	&=2\im\left(\sum_{|p_j|, |r_j|\leq K/3 \atop j \in E } (1+|p_i|^2)^k\sum_{|p_j|, |r_j|\leq K/3 \atop j \not\in E }   v_{\xi, p_E+r_{E^c}}v_{\xi, p_{E^c}+r_E}\overline{v}_{\xi, r}\overline{v}_{\xi, p} \right)\\
	&=2\im\left(\sum_{|p_j|, |r_j|\leq K/3 \atop j \in E } (1+|p_i|^2)^k\left(\sum_{|r_j|\leq K/3 \atop j \not\in E }   v_{\xi, p_E+r_{E^c}}\overline{v}_{\xi, r}\right) \left( \sum_{|p_j|\leq K/3 \atop j \not\in E }  v_{\xi, p_{E^c}+r_E}\overline{v}_{\xi, p}\right) \right)\\
	&=2\im\left(\sum_{|p_j|, |r_j|\leq K/3 \atop j \in E } (1+|p_i|^2)^k\left(\sum_{|t_j|\leq K/3 \atop j \not\in E }   v_{\xi, p_E+t_{E^c}}\overline{v}_{\xi, r_E+ t_{E^c}}\right) \left( \sum_{|t_j|\leq K/3 \atop j \not\in E }  v_{\xi, r_E+t_{E^c}}\overline{v}_{\xi,p_E+t_{E^c}}\right) \right).
	\end{align} 
Finally, 
	\begin{align}
	&\left(\sum_{|t_j|\leq K/3 \atop j \not\in E }   v_{\xi, p_E+t_{E^c}}\overline{v}_{\xi, r_E+ t_{E^c}}\right) \left( \sum_{|t_j|\leq K/3 \atop j \not\in E }  v_{\xi, r_E+t_{E^c}}\overline{v}_{\xi,p_E+t_{E^c}}\right)= \left(\sum_{|t_j|\leq K/3 \atop j \not\in E }   v_{\xi, p_E+t_{E^c}}\overline{v}_{\xi, r_E+ t_{E^c}}\right) \overline{\left( \sum_{|t_j|\leq K/3 \atop j \not\in E }  \overline{v}_{\xi, r_E+t_{E^c}}v_{\xi,p_E+t_{E^c}}\right)}.
	\end{align}

Therefore, the expression $\left(\sum_{(p, q, r, s) \in \Gamma_0^{K/3}}   (1+|p_i|^2)^k v_{\xi, q}v_{\xi, s}\overline{v}_{\xi, r}\overline{v}_{\xi, p} \right)$ is real-valued, and we are finished.
\end{proof}

\bibliography{waveguide}

\begin{thebibliography}{10}

\bibitem{B34}
A.~S. Besicovitch.
\newblock Sets of fractional dimensions (iv): on rational approximation to real
  numbers.
\newblock {\em Journal of the London Mathematical Society}, 1(2):126--131,
  1934.

\bibitem{B99}
J.~Bourgain.
\newblock Global wellposedness of defocusing critical nonlinear
  {S}chr\"{o}dinger equation in the radial case.
\newblock {\em J. Amer. Math. Soc.}, 12(1):145--171, 1999.

\bibitem{BD15}
J.~Bourgain and C.~Demeter.
\newblock The proof of the {$l^2$} decoupling conjecture.
\newblock {\em Ann. of Math. (2)}, 182(1):351--389, 2015.

\bibitem{B08}
R.~W. Boyd.
\newblock {\em Nonlinear optics}.
\newblock Elsevier/Academic Press, Amsterdam, third edition, 2008.

\bibitem{CF12}
R.~Carles and E.~Faou.
\newblock Energy cascades for {NLS} on the torus.
\newblock {\em Discrete Contin. Dyn. Syst.}, 32(6):2063--2077, 2012.

\bibitem{C57}
J.~W.~S. Cassels.
\newblock {\em An introduction to {D}iophantine approximation}.
\newblock Cambridge Tracts in Mathematics and Mathematical Physics, No. 45.
  Cambridge University Press, New York, 1957.

\bibitem{CKSTT08}
J.~Colliander, M.~Keel, G.~Staffilani, H.~Takaoka, and T.~Tao.
\newblock Global well-posedness and scattering for the energy-critical
  nonlinear {S}chr\"{o}dinger equation in {$\Bbb R^3$}.
\newblock {\em Ann. of Math. (2)}, 167(3):767--865, 2008.

\bibitem{CKSTT10}
J.~Colliander, M.~Keel, G.~Staffilani, H.~Takaoka, and T.~Tao.
\newblock Transfer of energy to high frequencies in the cubic defocusing
  nonlinear {S}chr\"{o}dinger equation.
\newblock {\em Invent. Math.}, 181(1):39--113, 2010.

\bibitem{DGG17}
Y.~Deng, P.~Germain, and L.~Guth.
\newblock Strichartz estimates for the {S}chr\"{o}dinger equation on irrational
  tori.
\newblock {\em J. Funct. Anal.}, 273(9):2846--2869, 2017.

\bibitem{D12}
B.~Dodson.
\newblock Global well-posedness and scattering for the defocusing,
  {$L^{2}$}-critical nonlinear {S}chr\"{o}dinger equation when {$d\geq3$}.
\newblock {\em J. Amer. Math. Soc.}, 25(2):429--463, 2012.

\bibitem{GG10}
P.~G\'{e}rard and S.~Grellier.
\newblock The cubic {S}zeg\"{o} equation.
\newblock {\em Ann. Sci. \'{E}c. Norm. Sup\'{e}r. (4)}, 43(5):761--810, 2010.

\bibitem{GG22}
F.~Giuliani and M.~Guardia.
\newblock Sobolev norms explosion for the cubic {NLS} on irrational tori.
\newblock {\em Nonlinear Anal.}, 220:Paper No. 112865, 25, 2022.

\bibitem{GPT16}
B.~Gr\'{e}bert, \'{E}. Paturel, and L.~Thomann.
\newblock Modified scattering for the cubic {S}chr\"{o}dinger equation on
  product spaces: the nonresonant case.
\newblock {\em Math. Res. Lett.}, 23(3):841--861, 2016.

\bibitem{GK15}
M.~Guardia and V.~Kaloshin.
\newblock Growth of {S}obolev norms in the cubic defocusing nonlinear
  {S}chr\"{o}dinger equation.
\newblock {\em J. Eur. Math. Soc. (JEMS)}, 17(1):71--149, 2015.

\bibitem{HPTV15}
Z.~Hani, B.~Pausader, N.~Tzvetkov, and N.~Visciglia.
\newblock Modified scattering for the cubic {S}chr\"{o}dinger equation on
  product spaces and applications.
\newblock {\em Forum Math. Pi}, 3:e4, 63, 2015.

\bibitem{HTT14}
S.~Herr, D.~Tataru, and N.~Tzvetkov.
\newblock Strichartz estimates for partially periodic solutions to
  {S}chr\"odinger equations in {$4d$} and applications.
\newblock {\em J. Reine Angew. Math.}, 690:65--78, 2014.

\bibitem{HPSW21}
A.~Hrabski, Y.~Pan, G.~Staffilani, and B.~Wilson.
\newblock Energy transfer for solutions to the nonlinear {S}chr\"{o}dinger
  equation on irrational tori.
\newblock {\em arXiv preprint arXiv:2107.01459}, 2021.

\bibitem{KM06}
C.~E. Kenig and F.~Merle.
\newblock Global well-posedness, scattering and blow-up for the
  energy-critical, focusing, non-linear {S}chr\"{o}dinger equation in the
  radial case.
\newblock {\em Invent. Math.}, 166(3):645--675, 2006.

\bibitem{K24}
A.~Khintchine.
\newblock Einige {S}\"{a}tze \"{u}ber {K}ettenbr\"{u}che, mit {A}nwendungen auf
  die {T}heorie der {D}iophantischen {A}pproximationen.
\newblock {\em Math. Ann.}, 92(1-2):115--125, 1924.

\bibitem{KV16}
R.~Killip and M.~Visan.
\newblock Scale invariant {S}trichartz estimates on tori and applications.
\newblock {\em Math. Res. Lett.}, 23(2):445--472, 2016.

\bibitem{L19}
G.~Liu.
\newblock Modified scattering for the cubic {S}chr\"{o}dinger equation small
  data solution on product space.
\newblock {\em SIAM J. Math. Anal.}, 51(5):4023--4073, 2019.

\bibitem{SW20}
G.~Staffilani and B.~Wilson.
\newblock Stability of the cubic nonlinear {S}chrodinger equation on an
  irrational torus.
\newblock {\em SIAM J. Math. Anal.}, 52(2):1318--1342, 2020.

\bibitem{TV16}
N.~Tzvetkov and N.~Visciglia.
\newblock Well-posedness and scattering for nonlinear {S}chr\"{o}dinger
  equations on {$\Bbb{R}^d\times\Bbb{T}$} in the energy space.
\newblock {\em Rev. Mat. Iberoam.}, 32(4):1163--1188, 2016.

\end{thebibliography}
\bibliographystyle{plain}
\end{document}